\documentclass[preprint,11pt]{elsarticle}

\usepackage{amssymb}

\usepackage{graphicx}
\usepackage{epstopdf}
\usepackage{algorithm}
\usepackage{algorithmic}
\ifpdf
  \DeclareGraphicsExtensions{.eps,.pdf,.png,.jpg}
\else
  \DeclareGraphicsExtensions{.eps}
\fi

\usepackage{orcidlink}
\usepackage{nicefrac}
\usepackage{caption}
\usepackage{siunitx} 
\usepackage{booktabs} 
\usepackage{multirow}
\usepackage{multicol}

\usepackage{orcidlink}
\usepackage{geometry}
\usepackage{amsfonts}
\usepackage{amsmath}
\usepackage{amsthm}
\usepackage{mathtools}
\usepackage{stmaryrd} 
\allowdisplaybreaks

\usepackage{ wasysym }
\usepackage{csquotes}
\usepackage{commath}
\usepackage[capitalise,noabbrev]{cleveref} 
\usepackage{soul}   

\newtheorem{theorem}{Theorem}
\newtheorem{remark}{Remark}

\newcommand{\review}[2][black]{%
{\color{#1}{#2}}%
}

\newcommand{\R}{\mathbb{R}}

\newcommand{\N}{\mathbb{N}}
\newcommand{\CC}{\mathcal{C}}
\newcommand{\prox}{\mathrm{prox}}
\DeclareMathOperator*{\argmin}{arg\,min}

\newcommand{\Acal}{\mathcal{A}}
\newcommand{\AcalV}{\Acal_{\mathrm{v}}}	
\newcommand{\AcalF}{\Acal_{\mathrm{f}}}	
\newcommand{\nn}{\mathsf{n}}

\newcommand{\alphaV}{\alpha_\mathrm{v}}
\newcommand{\alphaF}{\alpha_\mathrm{f}}
\newcommand{\sV}{s_{\mathrm{v}}}
\newcommand{\sF}{s_{\mathrm{f}}}
\newcommand{\sS}{\mathsf{S}}
\newcommand{\sI}{\mathsf{I}}
\newcommand{\sR}{\mathsf{R}}
\newcommand{\rS}{\rho_{\mathsf{S}}}
\newcommand{\rI}{\rho_{\mathsf{I}}}
\newcommand{\rR}{\rho_{\mathsf{R}}}

\newcommand{\hrS}{\widehat{\rho}_{\mathsf{S}}}
\newcommand{\hrI}{\widehat{\rho}_{\mathsf{I}}}
\newcommand{\hrR}{\widehat{\rho}_{\mathsf{R}}}
\newcommand{\rSo}{\rho_{\mathsf{S},0}}
\newcommand{\rIo}{\rho_{\mathsf{I},0}}
\newcommand{\rRo}{\rho_{\mathsf{R},0}}

\newcommand{\qS}{q_{\mathsf{S}}}
\newcommand{\qI}{q_{\mathsf{I}}}
\newcommand{\qR}{q_{\mathsf{R}}}
\newcommand{\hS}{h_{\mathsf{S}}}
\newcommand{\hI}{h_{\mathsf{I}}}
\newcommand{\hR}{h_{\mathsf{R}}}
\newcommand{\pS}{p_{\mathsf{S}}}
\newcommand{\pI}{p_{\mathsf{I}}}
\newcommand{\pR}{p_{\mathsf{R}}}
\newcommand{\DD}{\mathcal{D}}
\newcommand{\DDS}{\mathcal{D}_{\mathsf{S}}}
\newcommand{\DDI}{\mathcal{D}_{\mathsf{I}}}
\newcommand{\DDR}{\mathcal{D}_{\mathsf{R}}}

\newcommand{\Corner}{ \Delta_{\text{c}}^3 }

\newcommand{\llb}{\llbracket}
\newcommand{\rrb}{\rrbracket}

\DeclareMathOperator{\diag}{diag}
\DeclareMathOperator{\proj}{proj}

\journal{ }

\begin{document}

\begin{frontmatter}



\title{A Constrained Optimisation Framework for Parameter Identification of the SIRD Model\tnoteref{funds}}


\author[label_1]{Andrés Miniguano--Trujillo\corref{cor1}\orcidlink{0000-0002-0877-628X}}
\ead{Andres.Miniguano-Trujillo@ed.ac.uk}
\cortext[cor1]{Corresponding author}

\author[label_1,label_2]{John W. Pearson\orcidlink{0000-0002-6063-1766}}
\ead{j.pearson@ed.ac.uk}
\author[label_1,label_2]{Benjamin D. Goddard\orcidlink{0000-0002-8781-014X}}
\ead{b.goddard@ed.ac.uk}

\affiliation[label_1]{organization={Maxwell Institute for Mathematical Sciences, The University of Edinburgh and Heriot-Watt University}, 
            addressline={Bayes Centre}, 
            city={Edinburgh},
            state={Scotland},
            country={UK}}

\affiliation[label_2]{organization={School of Mathematics, The University of Edinburgh}, 
            addressline={James Clerk Maxwell Building}, 
            city={Edinburgh},
            state={Scotland},
            country={UK}}

\tnotetext[funds]{\textbf{Funding} A.M-T. acknowledges support from the MAC–MIGS CDT Scholarship under EPSRC grant EP/S023291/1. J.W.P. acknowledges support from EPSRC grant EP/S027785/1.}

\begin{abstract}
    We consider a numerical framework  tailored to identifying optimal parameters in the context of modelling disease propagation. Our focus is on understanding the behaviour of optimisation algorithms for such problems, where the dynamics are described by a system of ordinary differential equations associated with the epidemiological SIRD model. 
    Applying an optimise--then--discretise approach, we examine properties of the solution operator and determine existence of optimal parameters for the problem considered. Further, first--order optimality conditions are derived, the solution of which provides a certificate of goodness of fit, which is not always guaranteed with parameter tuning techniques.
We then propose strategies for the numerical solution of such problems, based on projected gradient descent, Fast Iterative Shrinkage--Thresholding Algorithm (FISTA), nonmonotone Accelerated Proximal Gradient (nmAPG),
and limited memory BFGS trust region approaches. We carry out a thorough computational study for a range of problems of interest, determining the relative performance of these numerical methods. 
    Our results provide insights into the effectiveness of these strategies, contributing to ongoing research into optimising parameters for accurate and reliable disease spread modelling. Moreover, our approach paves the way for calibration of more intricate compartmental models.
\end{abstract}



\end{frontmatter}



Compartmental models form a well--established approach for modelling a wide range of systems in mathematical biology \cite{Perthame2015,Garfinkel2017,Brauer2012}. In recent years, there has been a breadth of research into state--of--the--art compartmental models devised by mathematical biologists, e.g., \cite{Cheng2023,Elzinga2023,Moffett2023,Penn2023,Viscardi2023,CuevasMaraver2024,Barua2023}. In particular, a driver for the recent increase in the study of mathematical problems from epidemiology was the COVID--19 pandemic. We note many important contributions to this subject, e.g., \cite{behncke2000optimal, kantner2020beyond,Ketcheson2021,balderrama2022optimal,britton2023optimal,molina2022optimal,freddi2023infinite}. Typically, such approaches begin with a compartment model from mathematical biology describing the proportion of individuals in the system who belong to each of a certain set of classes, such as infected by a disease, recovered from that disease, or currently uninfected.

Assuming that the parameters for such a compartmental model are known, the dynamics and resultant predictions for mathematical biology and epidemiology are well--studied.  However, the identification of these parameters is a highly non--trivial task, and the dynamics of compartmental models are often highly sensitive to these values.  In principle, such parameters should be extracted from data, such as historical infection trends. There is some recent work in addressing this issue \cite{Bortz2023,Roy2024,Roy2022,Marino2008,Saldana2023}. A recent area of interest involves a related \emph{inverse problem}: given a particular target state (or ideal outcome) for the dynamics, can one systematically identify a set of system parameters that drive the biological system as close as possible to that state? 
For instance, we may wish to minimise some property of the system, such as the cumulative number of infected individuals over time, or the total number of deceased individuals after a given time. In practice, this cost may be balanced by one or more constraints, such as bounds on input variables (e.g., limited capacity for quarantine or hospital treatment), or an additional cost term beyond the closeness to the desired dynamics (e.g., total duration of lockdowns, or another measure of social dissatisfaction).

As inverse problems often exhibit more complicated dynamics than analogous forward problems, it is of crucial importance to understand how such problems can be tackled numerically. The main contribution of this work is to devise and examine the behaviour of optimisation algorithms for such systems, so that we may gather a keen understanding of how best to generate rapid and realistic simulations of epidemiological processes, with a view to identifying biological dynamics and making policy decisions. The examination of such algorithms for realistic applications can provide a springboard to resolving inverse problems arising from more complicated constraints, such as models with additional compartments, diffusive SIRD--based models, and nonlocal partial differential equations.

The model we shall consider in this paper is the standard Susceptible, Infected, Recovered, and Deceased model (SIRD) \cite{Kermack-1927,Bailey1975}, a system of three compartments whose evolution is governed by three ordinary differential equations (ODEs). Here, Susceptible individuals \((\rS)\) are infected at a rate \(\beta \rI\), where \(\beta\) is the effective transmission rate, and Infected persons  \((\rI)\) recover at rate \(\gamma\) and die at rate \(m\). Recovered persons  \((\rR)\) are understood to be immune to the disease. The governing equations are
\begin{equation}
	\od{\rS}{t} = -\beta \rS \rI,
	\qquad
	\od{ \rI }{t} = \beta \rS \rI - \gamma \rI - m\rI,
	\qquad
	\od{ \rR }{t} = \gamma \rI,
	\label{sys:SIRD}
\end{equation}
where the parameters \(\beta\), \(\gamma\), and \(m\) are understood to be non--negative rates in \([0,1]\). Notice that the scaling of this interval can change for each parameter if we consider a different time scale or normalise the equations in terms of population proportions. However, this is not a limitation for the analysis or applicability of the system. Moreover, we comment that the model applies only to timescales or infections where we can ignore the deaths of susceptible and recovered individuals; this is a standard assumption when using such models.

As the problem of key interest in this paper is the \emph{parameter identification} of such a model, we wish to determine the parameter values that best capture the epidemiological dynamics of the system \eqref{sys:SIRD}. 
The need of reliable data--driven methods for parameter identification became more apparent in the wake of the COVID--19 pandemic. 
For instance, in \cite{CuevasMaraver2024} a vaccination model is calibrated using structural identifiability and differential algebra \cite{Miao2011} to fit against data from fatalities in Switzerland and Andalusia.
In \cite{Barua2023} a vaccination model with immunisation is tuned with data from Hungary using the Latin hypercube sampling method, which is a common sampling method for parameter identification \cite{Roy2024}. 
The recent paper \cite{HarunOrRashidKhan2024} fits data from fatalities in Bangladesh 
against a variant of \eqref{sys:SIRD} (with additional equal death and birth rates) using a nonlinear least--squares approach through the default unconstrained MATLAB optimiser \texttt{fminsearch}. 
The authors then fix the fitted parameters and extend the model to include time--dependent controls to study the effects of face--mask usage, vaccinations, and testing rates. The controls are then determined using a forward--backward sweep
; see \cite[Chapter 4]{Lenhart2007}. 
In \cite{Petrica2023}, an averaged neural network is used to fit a regularised infection count of Romania, Hungary, Czech Republic, and Poland under the assumption of partial reporting. 
The neural networks are trained by simulating 
around 24 million
SIRD curves with random parameters selected inside user--guided intervals for a seven--day rollout; see also \cite{Dua2011,Pham2020}.
The prediction is used to forecast the disease based on a (time--dependent) piece--wise constant approximation of the parameters. Although an inverse problem (based on deceased count; see \cite{Petrica2022}) is analysed and proven to be well--posed, the data--driven approach is used instead for robustness.
Constant--valued parameters for \eqref{sys:SIRD} were calibrated in \cite{Fanelli2020} for the outbreaks in China and Italy. Here, the stochastic Differential Evolution heuristic \cite{Storn1997} was employed for determining a least--squares fitting which also considered some additional parameters associated with the initial condition. 
In \cite{Anastassopoulou2020}, an algebraic computation and MATLAB's implementation of the Levenberg--Marquardt algorithm \cite[\SS 2.2]{Oezisik2021} is used to fit the discrete SIRD dynamics.
Constant--in--time parameters are then used to provide a snapshot of the early evolution of the pandemic in China.
Time--dependent parameters that are constant across days are studied in \cite{Pacheco2021} using a function estimation approach.
The Levenberg--Marquardt method is employed with a Tikhonov regularisation parameter to analyse data from Italy and Brazil.
%
%
%
A constrained least--squares method validated with uncertainty quantification estimates is used for Netherlands and Spain transmission data in \cite{Smirnova2020}. The fitted model is a variant of \eqref{sys:SIRD} with sigmoidal contact rates influenced by fatality data \cite{Atkeson2020}. Here a predictor--corrector algorithm is used, combining all--at--once methods with the optimisation of a reduced cost functional inside the scope of alternate or block--coordinate minimisation \cite[Chapter 14]{Beck2017}. The approach is used to estimate a time--dependent disease contact rate while the other epidemiological parameters are fixed. We note that an age--structured variant of \eqref{sys:SIRD} is studied in \cite{Dutta2021}. Here, data from England and France are used to calibrate the model parameters using an Approximate Bayesian Computation (ABC) approach \cite{Beaumont2010}. We also highlight the wide applicability of parameter identification technologies for other problems in biology; see \cite{KahleLam2020,MiaoZhu2024} among many recent works.  

A common aspect of the aforementioned works is that the models are treated in a discrete setting. Clearly, this is understandable from the point of view of data collection as epidemiological data is reported on a periodic basis. 
However, it is desirable to treat the parameter calibration problem using a numerical approach that reflects the infinite--dimensional structure of the original system, meaning we may have a concrete guide as to when we (approximately) reach an optima. 
This is the focus of the present work. 
We will tackle the parameter identification problem using an optimise--then--discretise approach, meaning we will work with the model in a functional setting and derive certificates of optimality that analytically describe an optimal calibration.
This approach has a two--fold advantage: we can evaluate the optimality of solutions from a continuous perspective while also preserving the structural properties expected from the continuous system (e.g., positivity, conservation of mass).
We will focus on determining constant--in--time parameters (a reasonable assumption for small periods of time like a wave of infections) and present some insights about the more traditional time--dependent parameter problem. 
The fact that the studied parameters must stay constant in time represents an additional challenge as full domain information is required for their determination. This means that nonlocal interactions of the model quantities govern the optimality conditions. As a result, tailored algorithms are needed to incorporate the nonlocal terms into the optimisation routines. By contrast, optimality systems for time--dependent parameters provide pointwise information about the parameters, allowing the use of standard all--at--once and sweeping methods.
Finally, we note that our approach is agnostic to data, so we can determine suitable parameters for any dataset that is suitable for \eqref{sys:SIRD}.

The core contributions of this work are the analysis of such a parameter identification problem, and the adaptation of robust numerical methods for obtaining optimal parameters for the continuous model, by treating this problem on the infinite--dimensional level. We apply a standard Projected Gradient Descent approach as a benchmark for our numerical study, and propose Fast Iterative Shrinkage--Thresholding Algorithm (FISTA), nonmonotone Accelerated Proximal Gradient (nmAPG), and non--convex limited-memory BFGS (LM--BFGS) to improve computational performance.

This paper is structured as follows. In \Cref{sec:main} we state the optimisation problem of interest, analyse properties of the solution operator, and demonstrate the existence of an optimal control. In \Cref{sec:opt}, we derive necessary optimality conditions. In \Cref{Sec:Algorithms}, we outline the Projected Gradient Descent, FISTA, nmAPG, and LM--BFGS algorithms to be applied. In \Cref{sec:numerics}, we provide a systematic numerical study of optimisation algorithms applied to such epidemiological systems, centred on several problems of practical relevance, and in \Cref{sec:conclusions}, we provide some concluding remarks.

\newpage
\section{Optimisation problem}
\label{sec:main}

Let \( T \in \R^+\) be the final time and \(\rho_0 = \begin{psmallmatrix} \rho_{\mathsf{S},0} & \rho_{\mathsf{I},0} & \rho_{\mathsf{R},0} \end{psmallmatrix}^\top \in \R^3_{\geq 0}\) be the initial state.
Denote by \( \Acal \coloneqq \AcalV \times \AcalF \) the space of parameters. 
The set \( \AcalV\) consists of parameters that are unknown to the modeller and may either be constant or time--dependent. Conversely, \( \AcalF \) contains parameters that are always fixed. This construction allows us to examine the parameter identification problem driven by the dynamics of \eqref{sys:SIRD} whenever \(\beta\), \(\gamma\), or \(m\) are fixed, constant, or time--dependent. Let us further denote by \( \sV \) and \( \sF \) the number of parameters represented by \(\AcalV\) and \(\AcalF\), respectively.

We can write \eqref{sys:SIRD} as the system of ordinary differential equations
\begin{equation}
\label{eq:ode-sys}
	\od{\rho}{t} = f(\rho,\alpha,t)	\qquad \text{for } t\in (0,T),
    \qquad\text{and}\qquad \rho(0) = \rho_0.
\end{equation}
Here, \(\alpha\) is a vector of parameters, \(\rho\) is a vector field whose entries are identified as the triplet \( (\rS, \rI, \rR)^\top \hspace{-0.3em} \), and \(f: \R^3\times \Acal \times [0,T] \to \R^3\) is the function given by the right hand--side of each equation in \eqref{sys:SIRD}. Specifically, we have that
\(
	f(\rho,\alpha,t) \coloneqq
	\begin{pmatrix}
		-\beta \rS \rI & \beta \rS \rI - \gamma \rI - m \rI & \gamma \rI
	\end{pmatrix}^\top \hspace{-.3em}.
\)

Let \(\widehat{\rho} \in [L^2(0,T)]^3\) be a target state and \(\alphaF \in \AcalF\) be a vector of fixed parameters. Our goal is to find a trajectory \(\rho\), that is a response to the system \eqref{eq:ode-sys} controlled by \(\alphaV \in \AcalV\). 
We aim to minimise the following payoff functional of tracking type:
\begin{equation}
\label{eq:general_objective}
	J(\rho,\alphaV) \coloneqq \int\limits_0^T r( \rho, \alphaV ) \dif t,
\end{equation}
where \(r: \R^3 \times {\R^{\sV}} \to \R\) is a running payoff (for example, we will be interested in variants of \( r(\rho, \alphaV) = \frac{1}{2} \|\rho - \widehat{\rho} \|^2 \),
which is a standard choice for calibrating ecological parameters \cite{Saldana2023}). For simplicity, we do not include a terminal payoff in the cost functional \eqref{eq:general_objective} at this stage.

\begin{remark}
    Up to this point, the underlying function spaces for \(\Acal\) have been left unspecified, which is a modelling choice to allow for flexibility. For instance, we could optimise all the parameters in \eqref{sys:SIRD} as time--independent constants, meaning \( \AcalV = [0,1]^3\) and \( \AcalF = \varnothing\). Alternatively, we could optimise the contact rate as a time--dependent variable while retaining the other parameters fixed, yielding \( \AcalV = L^2(0,T)\) and \( \AcalF = [0,1]^2\). Other mixed options are also possible; however, we will assume that \( \AcalV\) is nonempty.
    
    In this paper, we will focus on the time--independent case,  with the exception of the numerical experiment in \Cref{sec:DD_PI}. Although it may seem a simplification of the framework, we will see that the conditions for optimising \eqref{eq:general_objective} require the evaluation of a nonlocal (integral) equation. As a result, constant--in--time parameter identification yields an additional complexity to be considered as we find discretised solutions that are not present in the pointwise characterisation of solutions for the time--dependent case.
\end{remark}

\subsection{Solution operator}

In this subsection, we will use scaled variables to obtain appropriate bounds, estimates, and qualitative information regarding the solution operator of system \eqref{sys:SIRD} and its corresponding derivative. This will allow us to show, under suitable conditions, that the optimal control problem under the payoff functional \eqref{eq:general_objective} is well posed. We will turn back to the original unscaled variables in the later sections.

Throughout this work, we will assume that (a) there is an initial population \(\nn \geq 1\) such that \( \nn = \sum \rho_0\), and (b) \(0 < \rSo, \rIo < \nn\), and \( \rRo \in [0, \nn)\).

If we scale \(\rho\) by \(\nn\), i.e., \(\rho = \nn \, \widetilde \rho\), we obtain the following system:
\begin{equation}
	\od{\rS}{t} = -[{\nn}{\beta}] \rS \rI,
	\qquad
	\od{ \rI }{t} = [{\nn}{\beta}] \rS \rI - \gamma \rI - m\rI,
	\qquad
	\od{ \rR }{t} = \gamma \rI.
	\label{sys:scaled-SIRD}
\end{equation}
Notice that here we have dropped the tilde decorator from \(\rho\) to represent its scaling by \(\nn\), and we will do so for the remainder of this section. 

Let us define the 3--dimensional corner of a cube given by
\(
	\Corner \coloneqq \big\{ \rho \in \R^3:\, \|\rho\|_1 \leq 1, \rho \geq 0 \big\}.
\)
In order for system \eqref{sys:scaled-SIRD} to be epidemiologically well posed in the sense of Hethcote \cite{Hethcote1976}, we expect that for any initial condition in the set \(\Corner\) we can obtain a solution of the initial value problem that stays inside \(\Corner\) for any later time \(t>0\). 
In other words, we expect that any vector field solution \(\rho\) of \eqref{sys:scaled-SIRD}, with non--negative initial values in \(\Corner\), remains non--negative for all later times, and the cumulative sum of the compartments will never exceed the maximum population density.
Let us discuss the existence of such vector fields, the first main result of this section, which follows from the compactness of \(\Corner\) and the regularity of the right hand--side of \eqref{sys:scaled-SIRD}.

\begin{theorem}
\label{th:state_existence}
	Given \(\alpha \in \Acal\) and \(\rho_0 \in \Corner\), there exists a unique solution \(\rho \in \mathcal{C}^\infty\) solving \eqref{sys:scaled-SIRD} inside a time interval of positive length.
\end{theorem}
\begin{proof}
	Let \(\alpha \in \Acal\) and \(\mathsf{f}(\rho,\alpha,\cdot)\) be defined by the right hand--side of \eqref{sys:scaled-SIRD}. We observe that \(\mathsf{f}\) defines an autonomous system of differential equations. Moreover, each entry of \(\mathsf{f}\) is a polynomial, hence \(\mathsf{f}\) is of class \(\mathcal{C}^1\). Moreover, we have that \(\Corner\) is compact and thus \( \mathsf{f} \) is Lipschitz in \(\Corner\) for any value of \(t\). As a result, there exists an interval \( I \), containing \(t=0\), for which there exists a unique solution for problem \eqref{sys:scaled-SIRD} \cite[Ch. 6 \S11]{Birkhoff1_1991}.
	
	Moreover, due to ODE regularity theory of analytic autonomous systems, we obtain that, as \(\mathsf{f}\) is of class \(\mathcal{C}^\infty\), then \(\rho\) is also of class \( \mathcal{C}^\infty\) on its interval of existence
	\cite[Ch. 6 \S9]{Birkhoff1_1991}.
\end{proof}

Global existence can be proved by determining that \(\Corner\) is an invariant region for \eqref{sys:scaled-SIRD}; i.e., any solution of \eqref{sys:scaled-SIRD} with starting point in \(\Corner\) stays within \(\Corner\) for all later times \(t > 0\). The ecological validity of the SIRD system is a direct consequence of this property. For this, let us recall the following theorem on invariant regions:

\begin{theorem}[Positively Invariant Regions, (\cite{Amann_1990}, Ch. 4 \S16)]
\label{Th:Invariant-Regions}
	Let \(G \in \CC^1(U,\R)\) such that \( \nabla G(u) \neq 0\) for all \( u \in G^{-1} (0)\); i.e., \(0\) is a regular value of \(G\). Then the region \( M \coloneqq \{ u\in U:\, G(u) \leq 0 \}\) is positively invariant for solutions of 
	the initial value problem \( \od{u}{t} = \mathsf{F}(u) \), \(u(0) = u_0\),
	 if and only if
	 \[
	 	\big\langle \nabla G(u), \mathsf{F}(u) \big\rangle \leq 0
		\qquad\qquad
		\forall u \in \partial M = G^{-1} (0).
	 \]
	 In particular, suppose \(0\) is a regular value of the family \( \{ G_i \}_{i \in \llb 1, k\rrb} \subset \CC^1 (U,\R) \), and we let \(M \coloneqq \displaystyle\bigcap_{i\in \llb1,k\rrb} \{ u\in U:\, G_i(u) \leq 0 \}\). If we have 
	 \[
	 	\big\langle \nabla G_i(u), \mathsf{F}(u) \big\rangle \leq 0
		\qquad\qquad
		\forall u \in \partial M = G_i^{-1} (0), \quad \forall i \in \llb 1,k\rrb,
	 \]
	 then \(M\) is positively invariant for solutions of the initial value problem with \(u_0 \in M\).
\end{theorem}

Using \Cref{Th:Invariant-Regions}, we arrive at the following result:

\begin{theorem}
\label{Th:ODE_Qualitative}
	The corner of a cube \(\Corner\) is a positively invariant region for the initial value problem \eqref{sys:scaled-SIRD} with \(\rho_0 \in \Corner\) and \(\alpha \in \Acal\).
\end{theorem}
\begin{proof}
Let us consider the following region \( \Sigma_0 = \{ \rho \in \R^3: \, \rho \geq 0 \}\). We claim that \(\Sigma_0\) is an invariant set. To see this, define \( G_{0, \sS} (\rho) \coloneqq -\rS\) and notice that \( \nabla G_{0, \sS} (\rho) = ( -1, 0, 0 )^\top\) never vanishes. Then
\[
	\big\langle \nabla G_{0, \sS}(\rho), \mathsf{f}(\rho) \big\rangle \big|_{\rS = 0} 
	= [{\nn}{\beta}] \rS \rI \big|_{\rS = 0} = 0,
\]
so we can conclude that \( \rS \geq 0\). Similarly, letting \( G_{0, \sI} (\rho) \coloneqq -\rI\) and noticing that any \(\rho\) with \( \rI = 0\) is an equilibrium, we derive that \( \rI \geq 0\). Finally, letting \( G_{0, \sR} (\rho) \coloneqq -\rR\), we similarly obtain \( \rR \geq 0\). By the theorem on positively invariant regions, we conclude that \(\rho \geq 0\) component--wise. In other words, we obtain that all the solutions of model \eqref{sys:scaled-SIRD} are nonnegative for nonnegative initial conditions.

Now let us consider the region with the upper bound \( \rho \leq 1\) and the additional constraint \( \|\rho\|_1 \leq 1 \), for which we define
\( G_1 (\rho) \coloneqq \rS + \rI + \rR - 1\). We then obtain
\[
	\big\langle \nabla G_{1}(\rho), \mathsf{f}(\rho) \big\rangle = \od{ (\rS + \rI + \rR) }{t} =  -m\rI \leq 0
\]
which holds as \( \rI \geq 0\). We can conclude that \( \rho\) is component--wise bounded from above; i.e., \( \rho \leq 1\). 

Therefore \(\Corner\) is positively invariant: an orbit that escapes \(\Corner\) would have to escape at some boundary point, and there is no boundary point at which that is possible. This, furthermore, implies that \eqref{sys:scaled-SIRD} has a global solution in \(\Corner\) for all \(T>0\).
\end{proof}

To finish this section, we will study some pointwise bounds, and some corresponding energy estimates, relating every choice of \(\alpha \in \Acal\) and the solution of the ODE system \eqref{sys:scaled-SIRD}. For this, we can just make use of Grönwall's inequality \cite[Ch. 1 \S\S10--11]{Birkhoff1_1991}, \cite[Appendix B]{Evans_2010}.

Consider the qualitative properties of a solution for the initial value problem \eqref{sys:scaled-SIRD} with \( \rho_0 \in \Corner\). 
The right hand--side of the equation governing the susceptibles compartment is always non--positive. As a result, using Grönwall's inequality, we obtain that
\(
	\rS(t) \leq \rSo .
\)
Integrating over \([0,T]\), we also obtain the energy estimate \( \|\rS\|_{L^2(0,T)}^2 \leq T \rSo^2  \). In what follows, we will simplify our notation and denote \( \| \cdot \|_{p} \) as the \(L^p(0,T)\) norm unless otherwise specified.
For the infected compartment, we notice that 
\(
	\od{ \rI }{t} \leq [{\nn}{\beta}] 
\)
holds since \( \rS \rI \leq \rI \leq 1\). Now, Grönwall's inequality tells us that
\(
	\rI(t) \leq \rIo +  [{\nn}{\beta}] t.
\)
Integrating over \([0,T]\), we also obtain the energy estimate
\(
	\| \rI \|_2^2 
	\leq
	\frac{3}{2} T \big( \rIo^2 +  [{\nn}{\beta}]^2 T^2 \big),
\)
which follows from Young's inequality \cite[Appendix B]{Evans_2010}. 
Using once more that \( \rI \leq 1\), we obtain  \( \od{\rR}{t} \leq \gamma\) and then
\(
	\|\rR\|_2^2 
	\leq
	\frac{3}{2} T \big( \rRo^2 +  \gamma^2 T^2 \big).
\)
Collecting all terms for each compartment, we obtain that
\[
	\|\rho\|_2^2 
	\leq 
	\frac{3}{2} T \big( \rSo^2 + \rIo^2 + \rRo^2 +  [{\nn}{\beta}]^2 T^2 + \gamma^2 T^2 \big);
\]
i.e., there is a polynomial constant \(C(\nn, T)\) depending only on \(\nn\) and \(T\) such that
\[
	\|\rho\|_2^2 \leq C(\nn, T) \big( \|\rho_0\|^2 + \|\alpha\|^2  \big).
\]
Notwithstanding, we can actually obtain a sharper bound from the fact that \(\rho \in \Corner\), which simply yields that \(\|\rho\|_2^2  \leq T\). However, the above inequality also allows us to state that \(\rho\) depends continuously on \(\alpha\) and the initial values \(\rho_0\).
Notice that if we were to allow \(\rho_0 = 0\), the above inequalities trivially hold since the solution of the associated initial value problem is just \(\rho_0\) itself.

Moreover, it is not difficult to see that the estimates \( \big\| \od{\rS}{t} \big\|_2^2 \leq T [{\nn}\beta]^2\) and \( \big\| \od{\rR}{t} \big\|_2^2 \leq T \gamma^2\) hold. 
The case for \(\od{\rI}{t}\) requires some additional care. First, observe that
\(
	\big| \od{\rI}{t} \big| \leq {\nn}{\beta}  + \gamma + m
\).
As a result, and by applying Cauchy's inequality, we obtain that 
\(
	\big\| \od{\rI}{t} \big\|_2^2 \leq 3 T
		\big( 
		[{\nn}{\beta}]^2  + \gamma^2 + m^2 
		\big).
\)
Adding up terms and using the fact that \(\nn \geq 1\), we obtain the following estimate for the time derivative:
\begin{equation}
\label{eq:l-2_estimate}
	\Big\| \od{\rho}{t} \Big\|_2^2 
	\leq 4T {\nn}^2  \|\alpha\|^2  .
\end{equation}

\subsection{Parameter sensitivity}

Sensitivity helps us understand how much parameter values influence the disease dynamics \cite[\S 6.5]{Martcheva2015}; see also \cite{Arriola2009,Saltelli2002}. 
It can further help us determine the parameters that can be controlled to reduce or eliminate infectious diseases \cite{Hussain2021,Alexe2009,Chitnis2008}. 
Here, we will analyse a particular quantity of interest that can be used to describe the expected evolution of a disease: the basic reproduction number. In what follows, and for clarity of presentation, we will assume that \( \gamma + m > 0\).

The basic reproduction number \(\mathcal{R}_0\) is a measure of the potential for disease spread
within a population. If \(\mathcal{R}_0 < 1\), then a few infected individuals in contact with a completely susceptible population will, on average, fail to replace themselves, and the disease will not spread. If, by contrast, \( \mathcal{R}_0 > 1\), then the disease will spread and cause an outbreak \cite{Driessche2008,Siriprapaiwan2018}.

The next--generation matrix method \cite{Driessche2008,Hurford2009} is an elegant technique that determines the reproduction number using information from the equilibria of a biological system. 
In the case of system \eqref{sys:scaled-SIRD}, the equilibria are characterised by the set 
\( E = \{ \rho \in \Corner: \, \rI = 0\} \).
Note that any point in \(E\) is a disease--free equilibrium (DFE). 
The next--generation form of \eqref{sys:scaled-SIRD} is given by the split of \(\rho\) into the variables \( x \coloneqq \rI \) and \( y \coloneqq \begin{psmallmatrix} \rS &  \rR \end{psmallmatrix}^\top \). The dynamics of \(x\) are given by 
\(
	\od{x}{t} = \mathcal{F}(x,y) - \mathcal{V}(x,y)
\),
where \(\mathcal{F}\) is the rate at which the infected cohort grows due to new secondary infections and \( \mathcal{V}\) is the rate of disease progression, death, and recovery decrease. Thus \( \mathcal{F}(x,y) = \nn \beta \rS \rI \) and \( \mathcal{V}(x,y) = (\gamma+m)\rI\).
The reproduction number is defined as the spectral radius of the next--generation matrix \( FV^{-1} \), where \( F \) and \(V\) are the matrices of partial derivatives of \( \mathcal{F}\) and \(\mathcal{V}\) with respect to \(x\) evaluated at a DFE, respectively. In particular, the basic reproduction number is determined by the special case where the initial population is a 100\% susceptible population.
As \(FV^{-1}\) is just a real number for system \eqref{sys:scaled-SIRD}, we obtain
\begin{equation*}
	\mathcal{R}_0 = \frac{\nn \beta \rS}{\gamma + m} \bigg|_{\rS = 1} =  \frac{\nn \beta}{\gamma + m}.
\end{equation*}

The (local) normalised sensitivity index (also known as elasticity) for a quantity \(Q\) with respect to a parameter \(h\) is defined by 
\(
	\Phi(Q|h) \coloneqq h \pd{ }{h} \log Q = \frac{h}{Q} \pd{Q}{h}
\)
(see \cite[\S 6.5]{Martcheva2015}). As \( \mathcal{R}_0\) determines the expected number of secondary infections produced by an index case in a completely susceptible population, we can analyse its sensitivity with respect to each parameter in \(\alpha\):
\begin{align}
	\Phi(\mathcal{R}_0|\alpha)
	&= 
	(\nabla \mathcal{R}_0) \review{\,\odot\,} \frac{ 1 }{ \mathcal{R}_0 } 
	\alpha
	=
	\begin{pmatrix}
		1 & -\frac{ \gamma }{\gamma + m} & -\frac{ m }{\gamma + m}
	\end{pmatrix}^\top \hspace{-0.3em} ,
 \label{eq:SIRD_Sensitivity}
\end{align}
\review{where \(\odot\) is the Hadamard product (i.e., component--wise product between vectors).}
\review{The values of \eqref{eq:SIRD_Sensitivity}}
tell us that \(\mathcal{R}_0\) is very sensitive to \(\beta\). 
Specifically, an increment of \(\beta\) results in the same increment in \(\mathcal{R}_0\). The values of \( \Phi(\mathcal{R}_0|\gamma) \) and \( \Phi(\mathcal{R}_0|m) \) are always negative, thus implying that any increase in these parameters results in a corresponding reduction of \(\mathcal{R}_0\). Clearly, having a high mortality rate is not desirable, thus optimising a recovery rate seems a natural choice. 

In the context of parameter identification, we can use any estimates on \( \Phi(\mathcal{R}_0|\alpha)\) to determine the sets \( \AcalV\) and \( \AcalF\): for instance, we could identify only \(\beta\) and \(\gamma\) for a disease with \emph{observed} low mortality rates or whenever mortality data is deemed precise. Our methodology, based on the optimise--then--discretise approach of optimal control, provides a general technique to tackle any selection. In what follows, we will analyse the parameter identification problem under a general scope for \(\Acal\) and will exemplify the procedure for the fixed parameter \( \AcalF = \{m\} \subset [0,1]\).

\subsection{Existence of an optimal control}\label{sec:Existence}

Let us restate our optimal control problem of interest:
\begin{subequations}
\label{Optimal_Control_Problem}
\begin{align}
	&\min_{\alphaV \in \AcalV} J(\rho,\alphaV) \coloneqq \int\limits_0^T r\big( \rho(t), \alphaV \big) \dif t,
	\\[-1em]
\intertext{\qquad\qquad such that} 
	&\dod{\rho}{t} = f\big(\rho, (\alphaV,\alphaF), t\big)	\qquad \text{for }\,\, t\in (0,T),
	\qquad\text{and}\qquad
	\rho(0) = \rho_0.
	\label{eq:OCP_ode-sys}
\end{align}
\end{subequations}
For every \( \rho \in \Corner\) and \(\alphaV\in \AcalV\), the running cost \(r\) is assumed to be convex, lower semicontinuous, and bounded from below. In particular, this is the case for the identification \(\ell^2\) choice \( r(\rho,\alphaV) = \frac{1}{2} \| \rho - \widehat{\rho} \|^2\).

\begin{theorem}\label{Th:Admissibility}
	The optimal control problem \eqref{Optimal_Control_Problem} admits a solution \(\alphaV^* \in \AcalV\).
\end{theorem}
\begin{proof}
	Since the functional \(J\) is bounded from below, there exists a minimising sequence \( (\alpha_{\text{v},n})_{n\in \N}\) in \(\AcalV\) such that \( J\big( \rho(\alpha_{\text{v},n}), \alpha_{\text{v},n} \big) \to \inf_{\alphaV \in \AcalV} J \). Moreover, this sequence is bounded by definition of \(\AcalV\); hence, by the Bolzano--Weierstrass theorem \cite[Theorem 3.4.8]{Bartle2011}, there exists a subsequence, still denoted by \( (\alpha_{\text{v},n})_{n\in \N}\), that converges strongly to a pair of parameters \(\alphaV^*\).
	
	Now, let \( \rho_n\) be the unique solution to the state system \eqref{eq:ode-sys}
	associated with \(\alpha_{\text{v},n}\), for each \(n\in \N\). From the qualitative results on \(\rho_n\) (see \Cref{Th:ODE_Qualitative}), we have that \( \rho_n \leq {\nn}\) pointwise; i.e., \( (\rho_n)_{n\in \N}\)  is uniformly bounded. 
	In a similar fashion to how we arrived at \eqref{eq:l-2_estimate}, we can obtain a uniform pointwise bound over \( (\rho_n')_{n\in \N}\) that only depends on \(\nn\) and \(m\). More generally, this can be carried out as well for any \(s\)--th derivative of \( (\rho_n)_{n\in \N}\). Calling upon Arzelà--Ascoli’s theorem \cite[Ch. 2, Lemma 7.2]{Amann_1990} (and nesting subsequences), there is a continuous function \(\rho^*\) which is the uniform limit of \( (\rho_n)_{n\in \N}\). It follows from the identity
	\(
		\rho_n(t) = \rho_0 + \int\limits_0^t f(\rho_n, (\alpha_{\text{v},n}, \alphaF), t) \dif t,
	\)
	and Lebesgue's dominated convergence theorem,
	that \( \rho^*\) satisfies the ODE system with parameters \((\alphaV^*,\alphaF)\), and that \( \rho_n' \to \rho^*\) as well. By the regularity of $f$ and the information obtained from the qualitative \Cref{th:state_existence,Th:ODE_Qualitative}, \(\rho^*\) is also of class \( \mathcal{C}^\infty\), and its range is contained in \(\Corner\).
	
	Finally, observe that by the integral semicontinuity of \(r\) \cite[Theorem 6.38]{Clarke_2013}, we obtain that
	\(
		J(\rho^*,\alphaV^*) \leq \liminf\limits_{n\to \infty} J(\rho_n,\alpha_{\text{v},n}) = \inf_{\alphaV \in \AcalV} J;
	\)
	i.e., \( (\rho^*,\alphaV^*)\) is a minimiser for the optimal control problem \eqref{Optimal_Control_Problem}.
\end{proof}


\section{Optimality conditions}\label{sec:opt}

The purpose of this section is to characterise local minimisers of the optimal control problem \eqref{Optimal_Control_Problem}. Moreover, from here onwards we will work with the unscaled system \eqref{sys:SIRD}.
First, let us comment about the differentiability of the solution map; i.e., the map \( S: \Acal \to \mathcal{C}^1(0,T)\) such that \( S(\alpha) = \rho \) solves the ODE constraints \eqref{eq:OCP_ode-sys}. 
This can be studied using the extended initial value problem
\begin{equation}
\label{sys:perturbation_parameters}
	\od{\rS}{t} = -{\omega_1} \rS \rI,
	\qquad
	\od{ \rI }{t} = {\omega_1} \rS \rI - \omega_2 \rI - \omega_3 \rI,
	\qquad
	\od{ \rR }{t} = \omega_2 \rI,
	\qquad
	\od{ \omega }{t} = 0,
\end{equation}
subject to the initial conditions \(\rho(0) = \rho_0\) and \( \omega(0) = \alpha\).
Observe that, by definition, solutions of the extended system will also be in \( \Corner \times \Acal\), thus \Cref{Th:ODE_Qualitative} holds for any initial conditions \( (\rho_0, \alpha) \in \Corner \times \Acal\). Taking this fact into consideration, as well as that each differential equation constraint of the extended system is an autonomous polynomial of the variables involved in the system, we can guarantee solutions for all times. The Perturbation Theorem \cite[Ch. 6 \S13]{Birkhoff1_1991} allows us to then conclude that the solution map \( S \) is not only differentiable with respect to \(\alpha\) but of any class \( \mathcal{C}^p\) with \(p\geq 1\). Moreover, the partial derivative \( \pd{S}{\alphaV}\)  satisfies the \emph{variational system}
\[
	\od{ }{t} \Big( \pd{S}{\alphaV} \Big) = \pd{f}{\rho} \pd{S}{\alphaV} + \pd{f}{\alphaV},
	\qquad \text{for } t\in (0,T),
	\qquad\text{and}\qquad
	\pd{S}{\alphaV}(0) = 0,
\]
which comes from formally differentiating the differential system and the initial conditions with respect to \(\alphaV\).

Having discussed the regularity of the solution map, we can now focus on establishing optimality conditions for the optimal control problem \eqref{Optimal_Control_Problem} using first--order information.
By the formal Lagrangian method, we now arrive at the following Pontryagin Maximum Principle:

\begin{theorem}\label{th:existence_adjoint}
Let us suppose that \(r\) admits a continuous derivative \( \pd{r}{\rho}\).
Let the process \( (\rho,\alpha_\mathrm{v} )\) be a local minimiser for the problem \eqref{Optimal_Control_Problem}.
Then, 
there exists an adjoint vector field \(q = \begin{psmallmatrix} \qS & \qI & \qR \end{psmallmatrix}^\top \hspace{-0.35em} \)
satisfying the following \textbf{adjoint system} for all \(t \in(0,T)\):
\begin{subequations}
\label{sys:adjoint}
\begin{align}
	\dod{\qS}{t} &= \beta \rI (\qS  - \qI ) - \dpd{r}{\rS}, 
	\\
	\dod{\qI}{t}  &= \beta \rS (\qS  - \qI) + \gamma (\qI  - \qR ) + m\qI - \dpd{r}{\rI}, 
	\\
	\dod{\qR}{t} &= -\dpd{r}{\rR},
\end{align}
\end{subequations}
with the transversality condition \(q(T) = 0\).
\end{theorem}
\begin{proof}
	The result follows from \cite[Theorem 22.2]{Clarke_2013} and noting that \(\alpha\) is constant across time. However, the same can be proven using optimisation theory in Banach spaces; see \cite[Proposition 9.2]{Manzoni2021}. In what follows, we will use the formal Lagrangian technique to derive the system \eqref{sys:adjoint}. The technique is a powerful tool that will allow us to determine a simplified expression for a derivative of the cost functional with respect to the parameter pair \(\alpha\).
	
	Let us start by defining the operators \( \DD: [H^1 (0,T)]^3 \times \Acal \to [L^2(0,T)]^3\) and \( G: [H^1 (0,T)]^3 \times \Acal \to [L^2(0,T)]^3 \times \R\), given by 
	\( 
		\DD(\rho,\alpha) \coloneqq \od{\rho}{t} - f(\rho,\alpha,t)
	\) and
	\( 
		G(\rho,\alpha) \coloneqq 
			\begin{pmatrix}
 				\DD(\rho,\alpha) & \rho(0) - \rho_0
			\end{pmatrix}^\top \hspace{-.3em}.
 \)
	In particular, any feasible solution of \eqref{Optimal_Control_Problem} satisfies \( G(\rho,\alpha) = 0\).
	
	Now, let us define the Lagrangian associated with the optimal control problem \eqref{Optimal_Control_Problem} as
\begin{align}
	\notag
	\mathcal{L}(\rho,\alphaV,q,p) &\coloneqq
	J(\rho,\alphaV)
	- \left\langle \begin{psmallmatrix} q \\ p \end{psmallmatrix}, G(\rho,\alpha) \right\rangle 
	\\		\label{eq:Lagrangian}
	&=
	\int\limits_{0}^T r(\rho,\alpha) - 
	\big[  \DDS(\rho,\alpha) \qS 	+ \DDI(\rho,\alpha) \qI	+ \DDR(\rho,\alpha) \qR \big]
	\dif t - \langle \rho(0) - \rho_0,  p\rangle.
\end{align}
Here \(q\) and \(p\) are Lagrange multipliers, which will be properly defined later on, and \( \DD_{(\jmath)}\) is the component of \(\DD\) associated with compartment \( {(\jmath)} \), and we understand that \( \alpha = (\alphaV,\alphaF) \).
From optimisation theory (see \cite[Theorem 9.2]{Manzoni2021}), we expect \( \pd{ }{\rho} \mathcal{L}\) to be equal to zero at any local minima.
Let us start by finding the directional derivatives of \eqref{eq:Lagrangian} with respect to \(\rho\), as these will allow us to find a characterisation of \(q\).

	For any sufficiently smooth function \(\hS\), 
integration by parts lets us obtain
\[
	\pd{ }{\rS} \bigg[\int\limits_{0}^T \DDS(\rho,\alpha) \qS  \dif t \bigg] (\hS)
	= 
	\hS \qS \Big|_0^T
	+ \int\limits_0^T -\od{\qS}{t} \hS + \beta \rI \hS \qS \dif t .
\]
Then we have that
\[
	\pd{ }{\rS} \mathcal{L} (\hS) = \int\limits_{0}^T \pd{r}{\rS} (\hS) 
	+ \od{\qS}{t} \hS - \beta \rI \hS \qS  + \beta \rI \qI \hS  \dif t
	- \hS(T) \qS(T) + \hS(0) \big[  \qS(0) - \pS \big]
\]
is equal to zero for any smooth \( \hS\) if and only if \( \pd{r}{\rS} 
	+ \od{\qS}{t}  - \beta \rI  \qS  + \beta \rI \qI = 0\), \( \qS(T) = 0\), and \( \pS = \qS(0)\).
As a result, we can define the following adjoint representation:
\[
	\pd{ }{\rS} \bigg[\int\limits_{0}^T \DD_S(\rho,\alpha) q_S  \dif t \bigg] (\hS)
	= 
	\int\limits_0^T \bigg[ -\od{\qS}{t} + \beta \rI \qS \bigg] \hS \dif t
	\eqqcolon 
	\left\langle \pd{ \DDS}{\rS}^* [\qS], \hS \right\rangle_{L^2(0,T)}.
\]
In a similar fashion, we can define the adjoint operators
\begin{align*}
	\pd{ }{\rI} \bigg[\int\limits_{0}^T \DDS(\rho,\alpha) \qS  \dif t \bigg] (\hI)
	&= 
	\int\limits_{0}^T  [ \beta \rS \qS ] \hI \dif x \dif t
	\eqqcolon 
	\left\langle \pd{ \DDS}{\rI}^* [\qS], \hI \right\rangle_{L^2(0,T)},
	\\
	\pd{ }{\rR} \bigg[\int\limits_{0}^T \DDS(\rho,\alpha) \qS  \dif t \bigg] (\hR) &= 0
	\eqqcolon 
	\left\langle \pd{ \DDS}{\rR}^* [\qS], \hR \right\rangle_{L^2(0,T)}
	,
\end{align*}
which hold for any \( \hI\) and \( \hR\) sufficiently smooth whenever \( \qI(T) = \qR(T) = 0\), \( \pI = \qI(0)\), and \( \pR = \qR(0)\).

At this point, let us define \( (\pd{ }{\rho} \DDS )^* (\rho,\alpha)\) as the vectorial field
\[
	 \Big( \pd{ }{\rho} \DDS \Big)^*(\rho,\alpha) [\qS] \coloneqq
	\begin{pmatrix}
		-\od{\qS}{t} + \beta \rI \qS
		&
		\beta \rS \qS
		& 
		0
	\end{pmatrix}^\top \hspace{-.3em}.
\]
Applying a similar argument as above with the derivatives of \(\DDI\) and \(\DDR\), we can define \( (\pd{ }{\rho}  \DDI)^* (\rho,\alpha)\) and \( (\pd{ }{\rho} \DDR)^*(\rho,\alpha)\) as the other two columns of the following adjoint operator matrix:
\begin{align*}
	\Big( \pd{ }{\rho} \DD \Big)^*(\rho,\alpha) [q]
	&\coloneqq
	\begin{pmatrix}
		-\od{\qS}{t} + \beta \rI \qS & -\beta \rI \qI   & 0 
		\\
		\beta \rS \qS & -\od{\qI}{t} - \beta \rS \qI + (\gamma+m) \qI  & -\gamma \qR 
		\\
		0 & 0 & -\od{\qR}{t}
	\end{pmatrix} 
	\\
	&= -\diag \big( \od{q}{t} \big)  - \big(\pd{f}{\rho}\big)^\top  \diag(q),
\end{align*}
where \( \pd{f}{\rho}\) is the Jacobian of \(f\) with respect to \(\rho\).

Since \( \big( \pd{ }{\rho} \mathcal{L} \big)^*(\rho,\alpha) [q] h = 0\) for any smooth \(h\) in \((0,T)\), and hence this holds as well in \(H^1(0,T)\), by the Fundamental Lemma of the Calculus of Variations, we have that \(q\) must satisfy the pointwise system
	\begin{equation}
	\label{eq:Functional_adjoint_1}
		0 = \pd{r}{\rho} - \Big( \pd{ }{\rho} \DD \Big)^*(\rho,\alpha) [q]  \, \mathbf{1}_3 ,
		\qquad\text{with}\qquad q(T) = 0,
	\end{equation}
	where \(\mathbf{1}_3\) is a constant three--dimensional vector of ones. Observe that the last expression is nothing else but
	\(
		\od{q}{t} =  -\big(\pd{f}{\rho}\big)^\top q - \pd{r}{\rho}
	\)
	with the terminal condition \(q(T) = 0\); i.e., \(q\) satisfies the adjoint system \eqref{sys:adjoint}. Observe that this is a non--homogeneous and linear variable--coefficient system of ordinary differential equations in \(q\). By continuity of \(\pd{r}{\rho}\) and the compactness of \(\Corner\) and \(\Acal\),
	there exists a unique \(\mathcal{C}^1\) solution for problem \eqref{sys:adjoint}
	\cite[Ch. 6 \S8]{Birkhoff1_1991}.
\end{proof}

Observe that as \(p\) is characterised completely by the values of \(q\) at any stationary point, we can restrict ourselves to the case \( p = q(0)\). Moreover, \eqref{eq:Functional_adjoint_1} can actually be expressed as \( \pd{G}{\rho}^* q = -\pd{J}{\rho}\) where the formal adjoint of the first derivative of \(G\) with respect to \(\rho\) also includes the condition \(q(T) = 0\). It turns out that if we define the reduced cost functional 
\(
	j(\alphaV) \coloneqq J( S(\alphaV),\alphaV),
\)
where \(S\) is the solution operator of the state system restricted to varying \(\alphaV\),
and consider the differentiability of \(J\) and \(S\), we have that \(j\) is a differentiable functional, and its derivative satisfies 
\begin{align}
	\notag
	\dpd{j}{\alphaV} (\alphaV) 
	&= \dpd{J}{\rho} \big( \mathcal{S}(\alphaV), \alphaV \big) \circ \dpd{\mathcal{S}}{\alpha}(\alphaV) + \dpd{J}{\alphaV} \big( \mathcal{S}(\alphaV), \alphaV \big)
	\\
	&= -\dpd{G}{\alphaV}\big( \mathcal{S}(\alphaV), \alpha \big)^* q
	+ \dpd{J}{\alphaV} \big( \mathcal{S}(\alphaV), \alphaV \big);
	\label{eq:reduced_derivative-th}
\end{align}
see \cite[Proposition 9.2]{Manzoni2021}. It is understood that in \eqref{eq:reduced_derivative-th}, \( q(0) \) is a value of \(q\).

Notice that there are no differential relationships for \(\alpha\) explicitly in \(G\), neither does the initial condition \(\rho_0\) depend on \(\alpha\). 
Hence, we obtain the following nonlocal operator:
\[
	\pd{G}{\alphaV } \big( {S}(\alphaV), \alpha \big)^* q = \int\limits_0^T  - \del{\dpd{f}{\alphaV} }^\top \hspace{-0.3em} q \dif t ,
\]
Using Leibniz's integral rule, we also obtain that 
\(
	\pd{J}{\alphaV} \big( S(\alphaV), \alphaV \big)
	=
	\int\limits_0^T \pd{r}{\alphaV} \big( S(\alphaV), \alphaV \big) \dif t,
\)
hence we obtain:
\begin{equation}
\label{eq:reduced_derivative}
	\pd{j}{\alphaV} (\alphaV) 
	= 
	\int\limits_0^T
	\del{\dpd{f}{\alphaV} }^\top \hspace{-0.3em} q 
	+ \pd{r}{\alphaV} \big( S(\alphaV), \alphaV \big)
	\dif t.
\end{equation}

Notice that we can also obtain \eqref{eq:reduced_derivative-th} and, consequently, \eqref{eq:reduced_derivative} by differentiating the formal Lagrangian \eqref{eq:Lagrangian} with respect to \(\alphaV\) and using Clairaut's theorem \cite{Christol1997}. 
Finally, observe that, since \(\alphaV\) is a constrained variable inside the convex set \(\AcalV\), stationarity gives a first--order necessary optimality condition for the gradient of the reduced objective \eqref{eq:reduced_derivative} at a local minimiser \(\alphaV\), which is given by the variational inequality
\(
	\Big\langle \pd{j}{\alphaV} (\alphaV), \omega - \alphaV \Big\rangle \geq 0
\)
for all \( \omega \in \AcalV\) \cite[Proposition 10.36]{Clarke_2013}. 
Using Hilbert's Projection Theorem \cite[Theorem 5.2]{Brezis_2011}, this is equivalent to the projection relationship
\begin{equation*}
	\alphaV = \proj_{\AcalV} \Big( \alphaV - \kappa \pd{j}{\alphaV} (\alphaV) \Big)
\end{equation*}
for all \(\kappa > 0\) \cite[Lemma 9.1]{Manzoni2021}, where \( \proj_{\AcalV} : \R^2 \to \AcalV \) is the coordinate--wise projection operator onto \( \AcalV\), that is, 
\begin{equation}
\label{eq:projection}
	(\proj_{\AcalV} \omega)_i = \max \big\{ 0, \min \{ \omega_i, 1 \}  \big\}.
\end{equation}
Consolidating all of our previous discussion, we have arrived at the following result:

\begin{theorem}
\label{th:Optimality-System}
	Let \( (\rho, \alphaV) \in [H^1(0,T)]^3 \times \AcalV \) be an optimal solution to problem \eqref{Optimal_Control_Problem}. There exists an adjoint state \(q \in [H^1(0,T)]^3\) such that the following optimality systems hold:

\noindent\textbf{State system} for all \(t \in (0,T)\) and initial condition \( \rho(0) = \rho_0\):
\begin{align*}
	\dod{\rS}{t} &= -\beta \rS \rI,
	\qquad
	\dod{\rI}{t} = \beta \rS \rI - \gamma \rI - m \rI,
    \qquad
	\dod{\rR}{t} = \gamma \rI.
\end{align*}
\noindent\textbf{Adjoint system}  for all \(t \in (0,T)\) and terminal condition \( q(T) = 0\):
\begin{align*}
    \begin{aligned}
	\dod{\qS}{t} &= \beta \rI (\qS  - \qI ) - \dpd{r}{\rS},
    \\
	\dod{\qI}{t} &= \beta \rS (\qS  - \qI) + \gamma (\qI  - \qR ) + m\qI - \dpd{r}{\rI},
    \qquad
    \dod{\qR}{t} = -\dpd{r}{\rR}.
    \end{aligned}
\end{align*}
\noindent\textbf{Gradient system} for all positive real values \(\kappa > 0\):
\vspace{-0.75em}
\begin{align*}
	\alphaV &= \proj_{\AcalV} \cbr{ \alphaV - \kappa
    \int\limits_0^T
	\del{\dpd{f}{\alphaV} }^\top \hspace{-0.3em} q 
	+ \dpd{r}{\alphaV} \big( S(\alphaV), \alphaV \big)
	\dif t
	}
    .
\end{align*}%
%
\end{theorem}
\vspace{-1\baselineskip}

\begin{remark}
    Observe that we have not specified which notion of derivative we are using whenever we differentiate with respect to \(\alpha\). This is a result from the equivalence of differentiability notions for Euclidean finite dimensional spaces \cite{Christol1997}.
\end{remark}

\begin{remark}
    Using regularity theory, \( \rho\) is actually of class \(\mathcal{C}^\infty\), and \(q\) possesses as many derivatives as \( \pd{r}{\rho}\).
\end{remark}

\begin{remark}
There is no closed analytical solution for the adjoint \(q\), but it can be expressed using the time ordered exponential \cite{Rugh-1996}  and the Peano--Baker series; see \cite[Theorem 3.3]{Rugh-1996} and \cite[Corollary 2.1]{Hartman_2002}. Let us also define the backward--in--time adjoint \(\bar{p}\) by the change of variables \(\tau \equiv T - t\) for which \( \bar{p}(t) \coloneqq q(T - t)\). Observe that \( \dif \tau = - \dif t\), then 
\( 
	\od{\bar{p}}{t} 
	= - \od{q}{t} (T-\tau) = \big(\pd{f}{\rho} (T-t) \big)^\top \bar{p} + \pd{r}{\rho} (T-t)
\). 
Let us define \( A(t) \coloneqq \big(\pd{f}{\rho} (T-t) \big)^\top\) and \( b(t) \coloneqq \pd{r}{\rho} (T-t)\), and note that \(\bar{p}(0) = 0\). This way we can write the adjoint system in the form \( \od{\bar{p}}{t} = A(t) \bar{p}(t) + b(t)\).
From \cite[Lemma 4.1]{Hartman_2002}, we obtain the estimate
\[
	\| \bar{p}(t) \|_{\ell^r} \leq \exp{ \bigg[ \int\limits_0^t \| A(s) \|_{\ell^r}  \dif s \bigg]} \int\limits_0^t \| b(s) \|_{\ell^r} \dif s.
\]
If we use the \(\ell^\infty\) induced norm, we can find that \( \| A(s) \|_{\ell^\infty} = \max\{2 \beta \rI, 2\beta  \rS + 2\gamma + m  \}|_{(T-s)} \leq 2[\nn \beta] + 2\gamma + m \), and then
\[
	\|q(t)\|_{\ell^\infty} \leq e^{ (2[\nn \beta] + 2\gamma + m) (T-t) } \int\limits_{t}^T \Big\| \pd{r}{\rho} (\tau) \Big\|_{\ell^\infty} \dif \tau.
\]
\end{remark}

\subsection{Choice of parameter space and cost}

We will consider running costs based on the following three functionals (scaled if appropriate) that are relevant in the context of fitting:
\begin{subequations}
\begin{align}
	r_1(\rho,\alpha) &= r_1(\rho) = \frac{1}{2} \| \rho - \widehat{\rho} \|^2
	=
	\frac{1}{2} \del{| \rS - \hrS |^2 + |\rI - \hrI|^2 + |\rR - \hrR|^2 },
	\label{eq:First-Objective-r}
	\\
	r_2(\rho,\alphaV) &= r_1(\rho,\alphaV)  + \frac{\theta}{2} \|\alphaV\|^2
	= r_1(\rho,\alphaV)  + \frac{\theta}{2} \sum_{i=1}^{\sV} \alpha_{\text{v}, i}^2
	& (\theta > 0)
	,
	\label{eq:Second-Objective-r}
	\\
	r_3(\rho,\alphaV) &= r_2(\rho,\alphaV) + \frac{\vartheta}{2T} \| \rho(T) -  \widehat{\rho}(T) \|^2_{\R^3}
	& (\vartheta,\theta > 0)
	.
	\label{eq:Third-Objective-r}
\end{align}
\end{subequations}
The tracking cost \(r_1\) is a continuous extension of the classical least--squares approach. It is particularly useful if the parameter identification problem is well posed. 
However, when uniqueness cannot be guaranteed, the regularised functional \(r_2\) allows us to guarantee the existence of a unique solution for some \( \theta > 0\); see for instance \cite{Lenhart1999}.
Moreover,
for a disease of concern, initial estimates of the population in each compartment might be prone to error.
Certainly, as time passes, better measurements are taken, hence it makes sense to also study the fitting at a final cutoff time, thus \( r_3\) contains a terminal cost associated with a final observation of the disease. Here we assume that \( \theta\) and \(\vartheta\) are positive.
Although we have not studied a terminal cost in the previous sections, the existence theory
developed holds by defining the non--zero terminal condition \( q(T) = \vartheta \big[ \rho(T) - \widehat{\rho} (T) \big] \). The boundedness of \(q\) still holds by \cite[Lemma 4.1]{Hartman_2002}.

\begin{remark}
As discussed in \cite{Joshi2004}, the regularised functional gives an approximation of the identification problem due to the nature of the regularisation itself. However, from the analysis of the state system, it can be shown that there is a minimising sequence \(\theta_n \to 0\), such that the corresponding sequence of associated optimal controls also converges to a solution of the identification problem. 
In practice, we have observed that black--box methods applied to the first order optimality conditions are able to find this solution without the need for a regularisation term.
\end{remark}

\section{Optimisation algorithms}\label{Sec:Algorithms}

The separability and solvability of the state and adjoint systems allows to use \eqref{eq:reduced_derivative} in any gradient--based method to solve the optimal control problem. Iterative (also known as black--box or reduced space) methods rely on the natural decomposition of the state and adjoint systems: Given a feasible control, an existing algorithm for the solution of the state system is embedded into an optimisation loop. To improve numerical performance, a gradient is incorporated by augmenting a routine for solving the adjoint system. Then we can apply any descent technique to update the control and repeat the process \cite{Manzoni2021}. The main drawback of these methods is that the solver for the state equation results in a costly step \cite{Herzog2010}. Another family of methods are all--at--once methods which treat the control and state variables as independent optimisation variables. We refer the reader to \cite{Herzog2010} for a general overview of these methods. 

Due to the nature of the nonlinear optimal control problem \eqref{Optimal_Control_Problem}, and the nonlocality of the reduced derivative \eqref{eq:reduced_derivative}, we find that reduced space methods are appealing. As a result, we decouple the updates of \(\alphaV\), \(\rho\), and \(q\) to obtain a descent algorithm which considers the box constraints. The nature of the bounds for \(\Acal\) implies the need to project onto the set of admissible controls, which is achieved by the projection operator \eqref{eq:projection}. In this section, we will use four reduced space algorithms to solve \eqref{Optimal_Control_Problem}. Two come from the convex optimisation family, namely Projected Gradient Descent (PGD) and the Fast Iterative Shrinkage--Thresholding Algorithm, the third is the nonmonotone Accelerated Proximal Gradient method, a non--convex extension of FISTA, while the fourth is a limited memory BFGS trust region scheme adapted to non--convex optimisation. We now introduce the modified algorithms in the context of optimal control.

As we will observe, the choice of PGD and FISTA reflects their simplicity in the sense that only efficient objective evaluations are required to iterate the algorithms. As a result, both algorithms can run for a large number of iterations in a matter of seconds before reaching a stopping criteria. The difference between the algorithms is that FISTA, by construction, provides a surrogate model of the reduced objective that can be used to find feasible solutions with low cost. This advantage is also included in nmAPG, which uses a correction step to compensate for the lack of convexity.
On the other hand, LM--BFGS is the most advanced technique that incorporates additional information from the non--convexity of the reduced objective at the expense of more expensive iterations with additional optimisation steps involved. However, our results show that this algorithm can often find near--optimal solutions efficiently and with a small iteration count. As a result, our experiments will confirm that the choice of LM--BFGS can be used to quickly solve the problem and obtain high--quality solutions.

\subsection{First--order methods}

Possibly one of the simplest routines in convex optimisation is the Projected Gradient Descent algorithm. For a given differentiable function \(j\) whose argument \(\alpha\) is constrained to a closed and convex region \(\Acal\), PGD aims to minimise \(j(\alpha)\) by taking repeated steps in the opposite direction of the gradient \(\nabla j\), also known as the steepest descent direction. For more details, see \cite{Tr_ltzsch_2010,De_los_Reyes_2015}.

PGD is part of the family of first--order proximal--based methods and yields linear convergence \cite{Combettes2011}. Convergence in the objective can be improved by additional higher--order information \cite{Beck2017} or by modifying the iterates with the addition of momentum and inertial steps \cite{Beck2009}. In particular, the Fast Iterative Shrinkage--Thresholding Algorithm exhibits quadratic convergence for the objective function and linear convergence in the iterates \cite{Chambolle2015}. FISTA considers the following assumptions:
\begin{itemize}
	\item The optimisation problem can be posed as a global problem \( \min \{F(\alpha) = j(\alpha) + \imath(\alpha): \, \alpha \in \R^d \}\) (in our context \(d\leq 3\)),
	\item the optimisation problem is solvable; i.e., attains a finite minimum with a finite minimiser,
	\item \(\imath\) is a continuous convex function which is possibly nonsmooth, and
	\item \(j\) is a smooth convex function of the type \(\mathcal{C}^{1,1}\); i.e., it is continuously differentiable with Lipschitz gradient.
\end{itemize}
This last assumption can be relaxed for local or unknown Lipschitz constants. This incurs an additional backtracking step. 

The general reasoning of the algorithm comes from noticing that the quadratic approximation (i.e., a surrogate function) of \(F(\omega)\):
\begin{equation}
\label{eq:surrogate}
	Q_L (\alphaV,\omega) \coloneqq j(\omega) + \big\langle \alphaV-\omega, \nabla j(\omega) \big\rangle + \frac{L}{2} \|\alphaV-\omega\|^2 + \imath(\alphaV)
\end{equation}
has a unique minimiser in \(\alphaV\) given by the proximity operator of \(L^{-1}\imath\), evaluated at nothing else than a gradient step:
\[
	P_L (\omega) \coloneqq 
	\argmin_{\alphaV} \left\{ L^{-1} \imath(\alphaV) + \frac{1}{2} \left\| \alphaV - \big( \omega - L^{-1} \nabla j(\omega) \big) \right\|^2 \right\}
	=
	\prox_{L^{-1} \imath} \big( \omega - L^{-1} \nabla j(\omega)  \big).
\]
For our context in optimal control, we identify the convex indicator function over the set \(\Acal\) as \( \imath\). Standard optimisation theory yields that the proximity operator of an indicator function is simply the projection operator of the indicated set; i.e., \( \prox_{\imath} = \proj_{\AcalV}\) as in \eqref{eq:projection}. Hence, we obtain the simplified operator
\[
	P_L (\omega) = \proj_{\AcalV} \big( w - L^{-1} \nabla j(\omega)  \big).
\]
The resulting numerical scheme, adapted from \cite{Beck2009} and \cite{Chambolle2015}, is presented in \Cref{alg:FISTA-BT}.

\begin{algorithm}[h]
\begin{small}
\begin{algorithmic}[1]
\STATE \textbf{Input:} Initial guess \((\alpha_{\mathrm{v},0}, \alphaF) \in \Acal\), \(L_0 > 0\), \(\eta > 1\), \( \omega_0 = \alpha_{\mathrm{v},0}\), \(\theta_0 = 1\), \(\nu > 2\).
\STATE Set \(k := 0\).
\WHILE{ stopping criteria has not been met }
	\STATE Solve the state system \eqref{eq:OCP_ode-sys} to obtain \(\rho_k \coloneqq \rho(\omega_k)\).
	\STATE Solve the adjoint system \eqref{sys:adjoint} to obtain \(q_k\).
	\STATE Evaluate the reduced gradient \eqref{eq:reduced_derivative} \(\nabla j(\omega_k)\).
	\STATE Find the smallest \(i_k \in \N\) such that \\  \label{alg:FISTA-BT-Surrogate}
	\(
		\begin{cases}
			L = \eta^{i_k} L_{k},
			\\
			j\big( P_L(\omega_k) \big) \leq Q_L \big( P_L(\omega_k), \omega_k \big).
		\end{cases}
	\)
	\STATE Update \(L_{k+1} = \eta^{i_k} L_{k}.\)
	
	\STATE Update \( \alpha_{\mathrm{v},k+1} = P_{L_{k+1}}(\omega_k). \)
	\STATE Update \( \theta_{k+1} = 1 + \nicefrac{k}{\nu}.\)
	\STATE Update \(\omega_{k+1} = \alpha_{\mathrm{v},k+1} + \frac{\theta_k - 1}{\theta_{k+1}} ( \alpha_{\mathrm{v},k+1} -  \alpha_{\mathrm{v},k} ). \)
	\STATE Increase \(k \to k+1\).
\ENDWHILE
\end{algorithmic}
\vspace{0.1em}
\caption{FISTA with backtracking}
\label{alg:FISTA-BT}
\end{small}
\end{algorithm}

An advantage of \Cref{alg:FISTA-BT} with respect to other gradient--based methods is that the backtracking step only requires one evaluation of the gradient. However, observe that the reduced objective function 
\(
	j(\alphaV) \coloneqq J( S(\alphaV),\alphaV)
\)
is not necessarily convex. This can be inferred from the fact that the solution map \(S\) is non--convex as \(\alphaV\) and \(\rho = S(\alphaV)\) interact nonlinearly in the state system \eqref{eq:OCP_ode-sys}. As a result, even though the function \(J(\rho,\alphaV)\) can be convex in both its arguments, its reduced form \(j(\alphaV)\) can have a non--convex structure.

\begin{figure}[htbp]
\centering
	\includegraphics[scale=0.42]{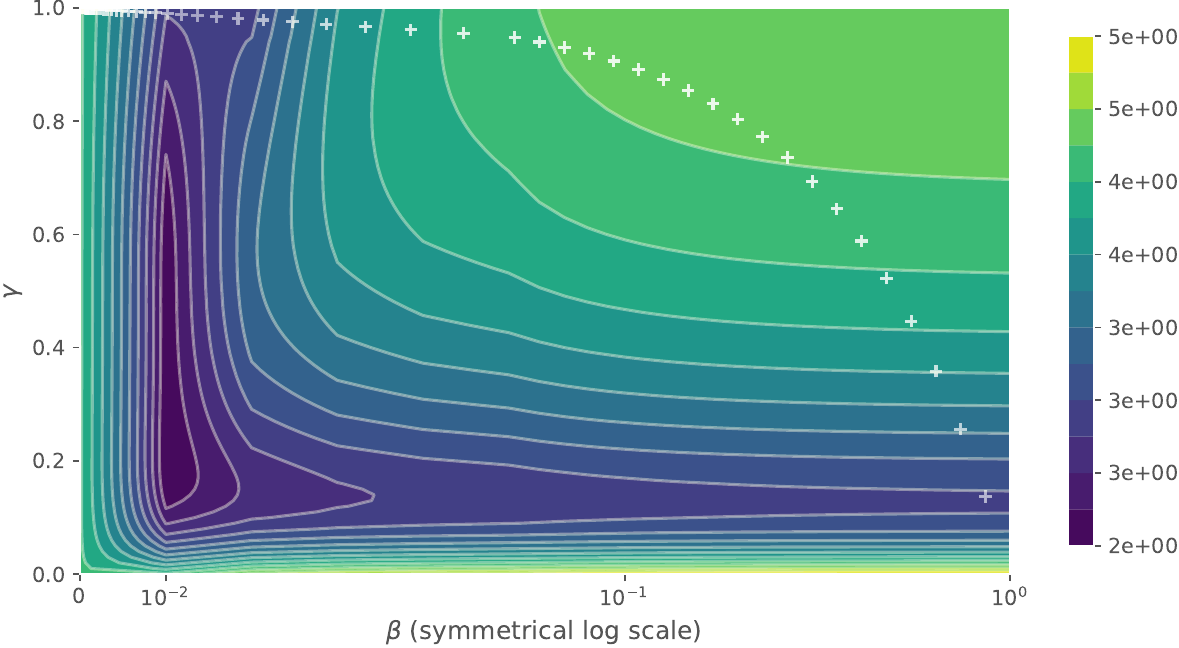}
    \includegraphics[scale=0.42]{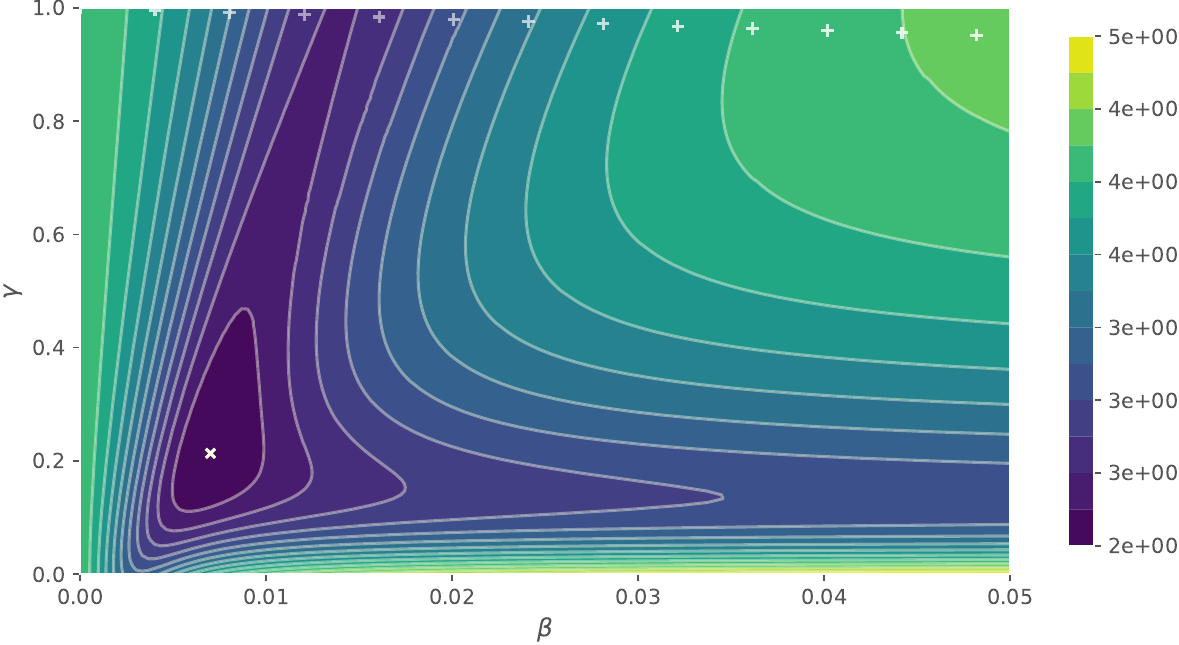}
	\vspace{1em}
\caption{Contour plots for the reduced objective \(j\) with zero target \(\widehat{\rho} = 0\). The panel on the left is the result of evaluating the scaled reduced objective function \(( \nicefrac{j(\alphaV)}{\nn^2} )\) over the set of feasible controls \(\AcalV = [0,1]^2\). The horizontal axis is plotted in symmetrical logarithmic scale for better visualisation of the minimal curves. The panel on the right zooms into the subregion \([ 0,\nicefrac{1}{20}] \times [0,1]\). The white \texttt{x} mark at \(\alpha \approx \begin{psmallmatrix} \mathtt{0.007} & \mathtt{0.213} \end{psmallmatrix}^\top \hspace{-0.3em} \) corresponds to the smallest objective found in this grid search. The \texttt{+} marks in both plots depict the non--convex slice \( j (\lambda \alpha_{\mathrm{v},1} + (1-\lambda)\alpha_{\mathrm{v},2} ) / {\nn^2} \) for \( \alpha_{\mathrm{v},1} = \begin{psmallmatrix} 1 & 0 \end{psmallmatrix}^\top \hspace{-0.3em} \), \( \alpha_{\mathrm{v},2} = \begin{psmallmatrix}0 & 1 \end{psmallmatrix}^\top \hspace{-0.3em}\), and different values of \(\lambda \in [0,1]\).}
\label{fig:Convex_Contours}
\end{figure}

At first glance, it would seem that FISTA is not the best algorithm for solving the optimal control problem \eqref{Optimal_Control_Problem}. Notwithstanding, the use of the surrogate function \(Q_L\) in \Cref{alg:FISTA-BT}, line \ref{alg:FISTA-BT-Surrogate}, works as a majorisation--minimisation step that considers a convex model of the reduced functional; see \eqref{eq:surrogate}. Hence, FISTA could be used in principle to find local minima in case that the reduced functional can be well approximated by a convex function. 
The previous example illustrates this idea to solve the optimal control problem \eqref{Optimal_Control_Problem}. The contour plots of the reduced objective in \Cref{fig:Convex_Contours} on the region that contains its minimisers shows the presence of a sharp minimiser\footnote{We say a point \(\alpha\) is a sharp minimum for the optimisation landscape \( \{f(\omega): \omega\in\Acal\} \) if \(\alpha\) is a minimiser and small changes around this point lead to significant changes in the function value.} inside a cone. Hence, a convex approximation of this cone can guide an optimisation routine to obtain such a minimiser.

The usage of such surrogate functions has inspired the application of proximal--based algorithms in the literature for non--convex optimisation; see \cite{Kanzow2022} for a review and \cite[Chapter 4]{Lin2020} for several extensions.
In particular, much focus has been put on extending the components of FISTA that accelerate the optimisation process: (i) the momentum or extrapolation sequence \( (\omega_k)_{k\in \N} \), (ii) the inertial weights \( (\theta_k)_{k\in \N} \), and (iii) the backtracking step that approximates the Lipschitz constant \(L\). 

An extension of FISTA for \(j\) non--convex and \(\iota\) convex (and possibly nonsmooth) is studied in \cite{Wen2017}. Here, the method recovers the restarted FISTA iteration \cite{ODonoghue2013,Alamo2019,Aujol2023} for the convex case under a suitable choice of stepsize. The method assumes that \(j\) can be decomposed as the difference of two convex functions.
The variant of \cite{Wright2009} considers a further step by letting \(\iota\) be non--convex but separable, thus generalising the indicator case of \cite{Birgin2000}. 
Here, the stepsize proposed by Barzilai and Borwein \cite{Barzilai1988} is employed to approximate Hessian information.

In \cite{Kanzow2022}, the convergence of proximal--based algorithms for the sum of a non--convex differentiable function and a lower semicontinuous function is analysed.
In particular, a variant of FISTA without momentum and inertia is analysed. The algorithm allows the use of a limited number of nonmonotone steps (i.e., allowing \(F\) to increase), while \(j\) is only assumed to be locally Lipschitz.
We refer the reader to \cite{Bolte2018,Cohen2021,DeMarchi2023}, and the references therein, for other approaches requiring just local Lipschitz continuity. In general, these methods use a monitoring sequence as part of a linesearch procedure but lack inertial acceleration. For our particular needs, we know \(j\) is Lipschitz in \(\AcalV\) due to the regularity of the solution map. However, we can extend \(j\) outside of \(\AcalV\) using a smooth partition of unity 
for the gradient, thus we can safely assume that \(j\) is globally Lipschitz. 
This enables the use of globalised accelerated schemes.

A general extrapolation version of the proximal gradient method is presented in \cite{Ghadimi2016}, also allowing for the use of inertia. Convergence of the algorithm and a randomised variant are analysed. However,
the acceleration steps are only valid under strong assumptions on the inertia weights that depend on a global Lipschitz constant \cite{Yang2023}.
Hence, \cite{Yang2023} proposes the use of nonmonotone steps considered with respect to a surrogate zero--order approximation of \(F\) given by the sum of \(F\) and a strongly convex function.
However, knowledge of the Lipschitz constant is still necessary.
In 
\cite{Li2015},
the nmAPG method is introduced. The algorithm uses a monitoring sequence consisting of a convex combination of historical objectives with exponentially decreasing weights. By combining extrapolation, inertia, and a correction step, the method requires at most two gradient evaluations per iteration. Moreover, its backtracking variant does not require any information from the Lipschitz constant and incorporates the Barzilai--Borwein approximation. 
The method is compared computationally with alternative approaches in \cite{Yang2023} and \cite{Liu2024}. In the former, comparable results are achieved for a class of least--squares problems. By contrast, the latter extension leads to a smaller iteration count at the expense of two additional gradient evaluations.

From this discussion, we have selected nmAPG with backtracking as our third first--order method. The steps of the method are presented in \cref{alg:nmAPG-BT}. We note that the extrapolation and correction steps are, in essence, two FISTA subiterations; see \cref{alg:FISTA-BT}.


\begin{algorithm}[h]
\begin{small}
\begin{algorithmic}[1]
\STATE \textbf{Input:} Initial guess \((\alpha_{\mathrm{v},0}, \alphaF) \in \Acal\),  $\alpha_{\mathrm{v},1} = \omega_1 = \nu_0 = \alpha_{\mathrm{v},0}$.
\STATE Select $\theta_0 = 0$, $\theta_1 = 1$, $\mu \in [0,1)$, $\delta >0$, $\eta > 1$, $c_1 = j(\alpha_{\mathrm{v},1})$, $\lambda_1 = 1$, \( 0 < \ell_{\min} \leq \ell_{\max} \).
\STATE Set \(k := 1\).
\WHILE{ stopping criteria has not been met }
	\STATE Set     \\[0.1em]
    \(
        \begin{cases}
        \nu_{k} = \alpha_{\mathrm{v},k} + \frac{\theta_{k-1}}{\theta_k} (\omega_k - \alpha_{\mathrm{v},k} ) + \frac{\theta_{k-1} -1 }{\theta_k} (\alpha_{\mathrm{v},k} - \alpha_{\mathrm{v},k-1})  ,
    \\
        s_{k} = \nu_k - \nu_{k-1},
        \qquad
        r_k = \nabla j(\nu_k) - \nabla j(\nu_{k-1}) ,
    \qquad
        L_k = \mathcal{P}_{[\ell_{\min}, \ell_{\max}]} \left( \nicefrac{s_k^\top r_k}{s_k^\top s_k} \right)  .
        \end{cases}
    \) \\[0.1em]
    
	\STATE Find the smallest \(i_k \in \N\) such that \\  
	\(
		\begin{cases}
			L = \eta^{i_k} L_{k},
            \qquad
            \omega_{k+1} = P_L(\nu_{k}) ,
			\\
			j(\omega_{k+1}) \leq \max\big\{c_k,j(\nu_k)\big\} - \delta \| \omega_{k+1} - \nu_k \|^2 .
		\end{cases}
	\)
	\IF{$j(\omega_{k+1}) \leq c_k - \delta \| \omega_{k+1} - \nu_k \|^2$}
        \STATE Set $\alpha_{\mathrm{v},k+1} = \omega_{k+1}$.
    \ELSE
        \STATE Set     
        \(
        s_{k} = \alpha_{\mathrm{v},k} - \nu_{k-1},
        \qquad
        r_k = \nabla j(\alpha_{\mathrm{v},k}) - \nabla j(\nu_{k-1})     ,
        \qquad
        L_k = \mathcal{P}_{[\ell_{\min}, \ell_{\max}]} \left( \nicefrac{s_k^\top r_k}{s_k^\top s_k} \right).
        \) \\[0.1em]
        
        \STATE Find the smallest \(i_k \in \N\) such that \\ 
	\(
		\begin{cases}
			L = \eta^{i_k} L_{k},
            \qquad
            \xi_{k+1} = P_L(\alpha_{\mathrm{v},k}),
			\\
			j(\xi_{k+1}) \leq c_k - \delta \| \xi_{k+1} - \alpha_{\mathrm{v},k} \|^2.
		\end{cases}
	\)
        \STATE Assign
        \( \alpha_{\mathrm{v},k+1} = 
        \begin{cases}
            \omega_{k+1} &\text{if }  j(\omega_{k+1}) \leq j(\xi_{k+1}),
                                    \\
                                    \xi_{k+1} &\text{otherwise}.
        \end{cases}
        \)
    \ENDIF
	
	\STATE Update $\theta_{k+1} = \frac{1}{2} \Big[ 1 + \sqrt{ 1 + 4 \theta_k^2 } \Big]$.
	\STATE Update $\lambda_{k+1} = \mu \lambda_k + 1$.
    \STATE Update $c_{k+1} = \lambda_{k+1}^{-1} \big[ \mu \lambda_k c_k + j(\alpha_{\mathrm{v},k+1}) \big]$.
	\STATE Increase \(k \to k+1\).
\ENDWHILE
\end{algorithmic}
\vspace{0.1em}
\caption{nmAPG with backtracking approach of \cite{Li2015}}
\label{alg:nmAPG-BT}
\end{small}
\end{algorithm}

\subsection{A quasi--Newton method}

The last solution technique that we will test is the active--set limited memory projected trust region algorithm for box constrained optimisation proposed in \cite{Yuan2014}. Let us discuss some preliminaries before presenting the method. In what follows, the only notation related to the problem \eqref{Optimal_Control_Problem} is the parameter vector \(\alpha\) and the reduced cost \(j\), and we assume for simplicity of exposition that \(\AcalV = \Acal\) and \(\AcalF=\varnothing\).

The limited memory BFGS method approximates the inverse of the Hessian of a functional $j: \R^n \to \R$ at iteration $k+1$, say $H_{k+1}$, without storing the dense matrices $H_k$ at each iteration. Instead, it stores $m_c \ll n$ correction pairs $ \{q_i, d_i\}_{i\in \{k-1, ..., k-m_c\}} \subset\mathbb{R}^{n,2} $, where
\(
    q_i \coloneqq \alpha_{i+1} - \alpha_{i}
\)
and
\(
    d_i \coloneqq \nabla j(\alpha_{i+1}) - \nabla j(\alpha_{i}),
\)
that contain information related to the curvature of $j$. In \cite{Byrd1994}, the authors introduce a compact form to define the limited memory matrix $ B_k = H_k^{-1} $ in terms of the $ n\times m_c $ correction matrices
\(
    S_k \coloneqq
        \begin{pmatrix}
            q_{k- m_c} & \cdots & q_{k-1}
        \end{pmatrix}
\) and
\(
    Y_k \coloneqq
        \begin{pmatrix}
            d_{k-m_c} & \cdots & d_{k-1}
        \end{pmatrix}.
\)
The matrix $S_k^\top Y_k$ can be written as the sum of the three matrices via
\(
    S_k^\top Y_k = L_k + D_k + R_k,
\)
where $L_k$ is strictly lower triangular, $D_k$ is diagonal, and $R_k$ is strictly upper triangular.

For $B_0 \coloneqq \bar{\theta} I_n$ with $\bar{\theta} > 0$, if the correction pairs satisfy $ q_i^\top d_i > 0 $, then the matrix obtained by updating $B_0$ with the BFGS formula and the correction pairs after $k$--times can be written as
\begin{equation}
\label{eq:LM-Update}
    B_k \coloneqq \bar{\theta} I_n - W_k M_k W_k^\top,
\end{equation}
where $W_k$ and $M_k$ are the block matrices given by
\vspace{-1em}
\begin{align*}
	W_k \coloneqq \begin{pmatrix}	Y_k & \bar{\theta} S_k \end{pmatrix}
	\qquad\text{and}\qquad
	M_k \coloneqq \begin{psmallmatrix}	-D_k & L_k^\top \\  L_k &  \bar{\theta} S_k^\top S_k	\end{psmallmatrix}^{-1} \hspace{-0.3em}.
\end{align*}
Note that, as $M_k$ is a $2 m_c \times 2 m_c$ matrix, the cost of computing the inverse in the last formula is negligible. Hence, using the compact representation for $B_k$, various computations involving this matrix become inexpensive, as is the case of the product of $B_k$ times a vector.
A similar representation to \eqref{eq:LM-Update} can be obtained for the inverse \(H_k\), and we refer the reader to \cite{Yuan2014} for further details.

Let \( u_b \in \R_{> 0}^3\) be the vector of natural upper bounds\footnote{The results of \cref{sec:Existence} assume \( u = \mathbf{1}_3\), but tighter bounds are allowed.} on \(\Acal\); i.e., \( \Acal = [0,u_1] \times [0,u_2] \times [0,u_3]\). 
The projection operator associated with the parameter set (understood entry--wise) becomes
\(
    \proj_{\Acal} \cdot = \max\big\{ 0, \min \{ \cdot, u_b \}  \big\}.
\)
Moreover, let us define the minimal diameter of \(\Acal\) by \( \ell_{\Acal} \coloneqq \min u_b\).

Now let us introduce the active set strategy. 
We introduce the quantity $\xi_k := \min\big\{\psi, c \|\nabla j(x_k)\|_2^{\zeta} \big\} $, where $\psi$, $c$, and \(\zeta\) are positive constants such that $\psi \in (0,\nicefrac{ \ell_{\Acal} }{2}) $, \(\zeta \in (0,1)\) \cite{Kanzow2001}, and define the strongly--active and inactive index sets by
\begin{subequations}
\label{ec:active_set_estimation}
\begin{align}
    A_k &\coloneqq \big\{ i \in \{1, \dots,n\}: \, \alpha_{k,i} 
            \leq \xi_k		 \lor 		\alpha_{k,i} \geq u_i -  \xi_k \big\},	
            \label{ec:active_set}
    \\
    I_k &\coloneqq \{1, \dots,n\} \setminus A_k = 
            \big\{ i \in \{1, \dots,n\}: \, \xi_k < \alpha_{k,i} < u_i -  \xi_k \big\},
    	\label{ec:inactive_set}
\end{align}
\end{subequations}
respectively, where $\alpha_{k,i}$ is the $i$--th element of $\alpha_k$.
Now, suppose the current trust region radius is $\widehat{\Delta} >0$, with its maximum value $\Delta_{\max} > 0$, and let $\kappa_k > 0$. We can obtain a search direction at step $\alpha_k$ as follows:

\noindent \textbf{Projected gradient direction:} Compute
\begin{equation}\label{ec:Yuan-proy_grad}
	d^{\text{G}}_{*k} (\widehat\Delta) := \proj_{\Acal} \left\{ \alpha_k - \dfrac{\widehat{\Delta} }{\Delta_{\max} }  \kappa_k \nabla j(\alpha_k) \right\}  - \alpha_k.
\end{equation}

\noindent \textbf{Projected trust--region direction:} We look for a direction $d^{\text{tr}}_{*k} (\widehat\Delta)$ defined for each index of the sets $A_k$ and $I_k$, respectively.
We begin with $A_k$, for which we let $v_k^{A_k}$ be the subvector
\[
    v_k^{A_k} \coloneqq
    \begin{cases}
        \alpha_{k,i}     & \text{if } \alpha_{k,i} \leq \xi_k,  \\
        u_i - \alpha_{k,i} & \text{if } \alpha_{k,i} \geq u_i - \xi_{k}.
    \end{cases}
\]
Then we define the subvector
\begin{equation}
    d^{A_k}_{*k} (\widehat\Delta)  \coloneqq \min \left\{ 1, \dfrac{ \widehat\Delta }{\|v_k^{A_k}\|} \right\} v_k^{A_k}.
    \label{ec:Yuan-SV}
\end{equation}
For the inactive set $I_k$ we solve a reduced trust--region subproblem. Here, let $B_k$ be partitioned into two submatrices $B_k^{A_k} $ and $B_k^{I_k} $, obtained by taking columns of $B_k$ indexed by $A_k$ and $I_k$, respectively. Let $ d^{I_k}_{*k} (\widehat\Delta) $ be a solution of the following TR–subproblem:
\begin{equation} \label{ec:subvector_inactive}
\begin{aligned}
        \min \;\, & d^\top \Big[ \big(B_k^{I_k}\big)^\top \big( \nabla j(\alpha_k) + B_k^{A_k} d^{A_k}_{*k}  \big) \Big] + \dfrac{1}{2} d^\top \big(B_k^{I_k}\big)^\top B_k^{I_k} d\\
\text{s.t. } & \|d\| \leq \widehat\Delta.
\end{aligned}
\end{equation}
The projected trust--region direction is then defined as
\begin{equation}
\label{eq:Yuan-TR_Direction}
	d^{\text{tr}}_{*k} (\widehat\Delta)
	\coloneqq 
    	\proj_{\Acal} \left\{ \alpha_k + \begin{pmatrix} d^{A_k}_{*k} (\widehat\Delta) \\ d^{I_k}_{*k} (\widehat\Delta)  \end{pmatrix}\right\}  - \alpha_k.
\end{equation}
Since this direction may not be a descent direction for $j$ for iterates away from the minimisers, we use a convex combination with the gradient direction as follows.

\begin{algorithm}[ht]
\begin{small}
\begin{algorithmic}[1]
\STATE \textbf{Input:} Initial guess \(\alpha_0 \in \Acal\), a symmetric positive definite matrix \(H_0\). Let constants satisfy \(\psi \in (0, \nicefrac{\ell_{\Acal}}{2}) \), \(c>0\), \(\zeta \in (0,1)\), 
\(0 < \nu_{\text{decrease}} < 1 < \nu_{\text{increase}}\), 
\(0 < \tau_{\text{accept}} < \tau_{\text{increase}} <1\), 
\(\sigma \in (0,1)\), 
\(\omega \in (0,1)\), 
\(\Delta_0 >0\), and 
\( \Delta_{\max} > \Delta_{\min} > 0\). 
Set \(m_c \in \mathbb{N}\) and \(B_0 = H_0^{-1}\).
\STATE Set \(k := 0\).
\WHILE{ stopping criteria has not been met }

	\STATE Let \(\Delta_k \coloneqq \min\big\{ \Delta_{\max}, \max\{\Delta_{\min}, \Delta_k\} \big\}\) and \(\widehat\Delta = \Delta_k\). 
	\label{alg:BFGS-Set-TR}
	
	\STATE \emph{Active--set estimation:} 
		Determine index sets \(A_k\) and \(I_k\) by \eqref{ec:active_set_estimation}.
		
	\STATE \emph{Trust--region subproblem:} 
		Find \(d_{*k}^{\text{tr}} (\widehat\Delta)\) by determining \(d_{*k}^{A_k} (\widehat\Delta )\) and \(d_{*k}^{I_k } (\widehat\Delta )\) as in \eqref{ec:Yuan-SV}, \eqref{ec:subvector_inactive}, and \eqref{eq:Yuan-TR_Direction}.
		\label{alg:BFGS-TR-Direction}
		
	\STATE \emph{Search direction:} Set
	\[
		\kappa_k 
		\coloneqq
		\min\left\{ 1, \dfrac{\Delta_{\max}}{\big\| \nabla j(\alpha_k) \big\|}, \dfrac{\omega}{\big\| \nabla j(\alpha_k) \big\|} \right\}.
	\]
	\STATE Compute \( d_{*k}^{\text{G}} (\widehat\Delta)\) and \(t_{*k}\) as in \eqref{ec:Yuan-proy_grad} and \eqref{ec:search_dir_one_dim}, respectively. Let
    \[
        d_{*k} (\widehat\Delta) 
        \coloneqq
        t_{*k} d_{*k}^{\text{G}} (\widehat\Delta) + (1- t_{*k}) d^{\text{tr}}_{*k} (\widehat\Delta).
    \]
    	
	\STATE \emph{Test the search direction:} Compute
	\[
		r_{*k} \coloneqq 
		\dfrac{ j(\alpha_k + d_{*k}) - j(\alpha_k)}{ \nabla j(\alpha_k)^\top d_{*k} (\widehat\Delta) + \frac{1}{2} \, d_{*k} (\widehat\Delta)^\top B_k d_{*k} (\widehat\Delta)}.
	\]
	
	\IF{ \(j(\alpha_k) - j \big(\alpha_k + d_{*k} (\widehat\Delta) \big) \geq -\sigma \nabla j (\alpha_k)^\top d^{\text{G}}_{*k} (\widehat\Delta)\) and \(r_{*k} \geq \tau_{\text{accept}}\) hold }
	\label{alg:LMBFGS-ConvexComb}
	
		\STATE Let \(q_k \coloneqq d_{*k}\), \(\alpha_{k+1} \coloneqq \alpha_k + d_{*k}\), and
		\[
			\Delta_{k+1} 
			\coloneqq
			\begin{cases}
				\hspace{1.5em} \widehat\Delta & \text{if } \tau_{\text{accept}} \leq r_{*k} < \tau_{\text{increase}},
				\\
				\nu_{\text{increase}}\, \widehat\Delta & \text{if } r_{*k} \geq \tau_{\text{increase}}.
			\end{cases}
		\]
		
		\STATE Let \(\widehat{m} := \min\{k+1,m_c\}\).
		\STATE Update \(B_k\) with the \(n \times \widehat{m}\) matrices \(S_k\) and \(Y_k\) to get \(B_{k+1}\).
		\STATE Let \(k \to k+1\) and return to line \ref{alg:BFGS-Set-TR}.
	
	\ELSE
		
		\STATE Let \(\widehat\Delta = \nu_{\text{decrease}} \widehat\Delta\)
		and return to line \ref{alg:BFGS-TR-Direction}.
	\ENDIF
\ENDWHILE
\end{algorithmic}
\vspace{0.1em}
\caption{Active--set Limited Memory BFGS Trust Region approach of \protect\cite{Yuan2014}}
\label{alg:LMBFGS}
\end{small}
\end{algorithm}

\noindent \textbf{Search direction:} Let
\(
	d_{* k} (\widehat\Delta)  
	\coloneqq
	t_{*k}\, d_{*k}^{\text{G}} (\widehat\Delta) + (1- t_{*k}) \, d^{\text{tr} }_{*k} (\widehat\Delta),
\)
where $t_{*k}$  is a solution of the following one--dimensional problem:
\begin{equation}
	\label{ec:search_dir_one_dim}
	\min_{s \in [0,1]} j\big( \alpha_k + s \, d_{*k}^{ \text{G} } (\widehat\Delta) + (1- s)\, d^{\text{tr} }_{*k} (\widehat\Delta) \big).
\end{equation}

With this background, we present the active--set limited memory approach of \cite{Yuan2014}, which we label LM--BFGS, in \Cref{alg:LMBFGS}. Notice that for every functional evaluation of \(j\), we have to solve the state system or forward problem \eqref{eq:OCP_ode-sys}, which can be done efficiently using an adaptive step method due to the regularity of the state; see \Cref{th:state_existence}. In principle, the most time consuming operations are finding \(t_{*k}\) in \eqref{ec:search_dir_one_dim} and computing line \ref{alg:LMBFGS-ConvexComb}. However, observe that the former is just a one--dimensional optimisation problem, for which we can use any out of the box optimiser, and, furthermore, we do not need an exact optimal solution, as line \ref{alg:LMBFGS-ConvexComb}
would reject any unacceptable step.

\subsection{Stopping criteria}

So far we have not discussed the stopping criteria of PGD and \Cref{alg:FISTA-BT,alg:nmAPG-BT,alg:LMBFGS}. Let \(\ell_\alpha\) be the length of \(\alpha_{\mathrm{v},k}\); i.e., \(\ell_\alpha \coloneqq \mathrm{length} (\alpha_{\mathrm{v},k})\). We will focus on three criteria: 
(a) iteration count; i.e., \(k \leq \mathrm{it_{\max}}\),
(b) the relative norm between iterates being less than a given tolerance:
	\(
		{ \|\alpha_{\mathrm{v},k+1} - \alpha_{\mathrm{v},k}\|_2 } < \mathtt{tol_a} {\sqrt{\ell_\alpha}},
	\)
(c) Himmeblau stopping rules; i.e., the absolute (relative) difference between two subsequent objectives being less than a given tolerance:
	\(
		|j(\alpha_{\mathrm{v},k+1}) - j(\alpha_{\mathrm{v},k})|  < \mathtt{tol_b}
    \)
    or
    \(
		|j(\alpha_{\mathrm{v},k+1}) - j(\alpha_{\mathrm{v},k})|  < \mathtt{tol_b}\, j(\alpha_{\mathrm{v},k}).
	\)
We will also report if the gradient satisfies the variational inequalities; i.e., \( [\nabla j(\alpha_{\mathrm{v},k})]_i = 0 \) if  \( [\alpha_{\mathrm{v},k}]_i     \in \mathrm{int}(\AcalV)\), \( [\nabla j(\alpha_{\mathrm{v},k})]_i \geq 0 \) if \( [\alpha_{\mathrm{v},k}]_i = \min \pi_i [\AcalV]\),  and \( [\nabla j(\alpha_{\mathrm{v},k})]_i \leq 0 \) if \( [\alpha_{\mathrm{v},k}]_i = \max \pi_i [\AcalV]\),
for all \( i \in \llb 1, \ell_\alpha \rrb \), and with \(\pi_i\) the canonical \(i\)--th projection operator and \( \mathrm{int}(\AcalV) \) the interior of \(\AcalV\).

\section{Numerical experiments}\label{sec:numerics}

In what follows we present the results of applying PGD and \Cref{alg:FISTA-BT,alg:nmAPG-BT,alg:LMBFGS} to the  optimal control problem \eqref{Optimal_Control_Problem}. In particular, we analyse three families of tests:
\begin{itemize}
	\item A given initial parameter vector \( \widehat{\alpha} \in \Acal\) is fixed, and the state system is evaluated to obtain \(\widehat{\rho}_a\). The aim is to fit the ODE model \eqref{eq:ode-sys} using the known data \(\widehat{\rho}_a\); hence, the optimal fitting is zero. Observe that \( \od{\widehat{\rho}_a}{t}\) is known solely by evaluating the state system. This test showcases the efficacy of each method when solving a problem with a unique known solution.
	
	\item We smoothly transform a solution of the state system, and then we average it on a finite number of intervals to obtain an upper semicontinuous function \(\widehat{\rho}_b \). This function mimics the rolling average methodology for epidemiological data. This test aims to study the influence of and robustness with respect to the regularisation parameter \(\theta\) against the given target.
	
	\item We present an application with data from the COVID--19 pandemic in Singapore. We also penalise a terminal condition within the objective function and allow time--dependent controls.
\end{itemize}

All the tests are run on a MacBook Pro 2020 M1 with 16 GB RAM. Both state and adjoint systems are solved using the \texttt{SciPy} explicit Runge--Kutta (RK) method of order 5(4) with default settings \cite{Virtanen2020}. The RK method may require the evaluation of both state and target at different time steps than the ones recovered after solving the state system. Thus, there is a need to interpolate the values of each state and target. As a result, we discretise time using Chebyshev points of the first kind, which help to reduce the effect of Runge's phenomenon \cite{Boyd2001,Berrut2004}. 
For the interpolation itself, we use a cubic Hermite interpolator spline for smooth data and linear Lagrange interpolation for discontinuous data. Finally, for the evaluation of the objective function, we use the composite Simpson’s rule.

\subsection{Known solution}

For this set of experiments, we fix the known parameter vector \( \alpha^* = \begin{psmallmatrix} \nicefrac{3}{100} & \nicefrac 3 5 & 0 \end{psmallmatrix}^\top \hspace{-0.3em} \), total population \(\nn = 200\) with the initial split \(\rho_0 = \begin{psmallmatrix} 199 & 1 & 0 \end{psmallmatrix}^\top \hspace{-0.3em} \), and final time \(T = 10\). 
The sensitivity \eqref{eq:SIRD_Sensitivity} associated to $\alpha^*$ is given by the vector \( \Phi(\mathcal{R}_0 | \alpha^*) = \begin{psmallmatrix} 1 & -1 & 0\end{psmallmatrix}^\top\). As a result, in what follows we consider the mortality rate as a fixed parameter; i.e., \( \AcalF = \{m=0\} \) and let \( \AcalV = [0,1]^2\). Thus, the gradient system reduces to
\begin{equation*}
	\alphaV = \proj_{\AcalV} \bigg\{ \alphaV - \kappa
	\int\limits_0^T
	\begin{psmallmatrix}	
				\rS \rI (\qI - \qS ) 	\\
		\phantom{\rS} \rI (\qR - \qI)
	\end{psmallmatrix} 
	+ \dpd{r}{\alphaV} \big( S(\alphaV), \alphaV \big)
	\dif t
	\bigg\}
.
\end{equation*}

In all tests, \(200\) Chebyshev points are used inside the interval \( [0,T]\), plus the two endpoints. The choice of \(\alpha^*\) and \(\rho_0\) yield the basic reproduction number \(\mathcal{R}_0 = \frac{[\nn \beta]}{\gamma + m} = 10\), which is greater than one and so is associated epidemiologically with an outbreak; i.e., the infected compartment attains a global maximum at some \(t \in (0,T)\) before the disease dies out \cite{Hethcote1976}. Finally, \(\alpha_{\mathrm{v},0}\) is always set to \((\mathtt{0.63696169}, \mathtt{0.26978671})^\top \hspace{-0.3em}\), which is a randomly generated vector. 

For the first--order methods, we use a discrepancy tolerance \( \mathtt{tol_a} = 10^{-7}\), Himmeblau's absolute tolerance \( \mathtt{tol_b} = \num{5e-13}\), and a maximum number of iterations \( \mathrm{it_{\max}} = 10\,000\). In LM--BFGS, we use the same Himmeblau criterion, but the discrepancy tolerance is replaced by the minimum trust region radius \(\Delta_{\min} = 10^{-6}\), and we choose \( \mathrm{it_{\max}} = 100\).
We do not control the discrepancy in LM--BFGS, as the distance between successive iterations is adapted by the active set estimation.

Observe that by the existence and perturbation results (see \Cref{th:state_existence} and \eqref{sys:perturbation_parameters}), two trajectories of the state system cannot overlap. As a result, beside the target state \( \widehat{\rho}_a\)  associated with the control \(\alpha^*\), there is no other solution to the optimal control problem \eqref{Optimal_Control_Problem} that can attain a zero fitting cost; thus, it is not necessary to use a regularisation term in the objective function. This justifies the choice of the following objective:
\begin{equation}
\label{eq:First-Objective}
	J(\rho,\alphaV) \coloneqq \frac{1}{2} \int\limits_0^T | \rS - \hrS |^2 + |\rI - \hrI|^2 + |\rR - \hrR|^2 \dif t.
\end{equation}

\Cref{tab:ODE-Optimisation} contains a summary of the results for PGD and \Cref{alg:FISTA-BT,alg:nmAPG-BT,alg:LMBFGS}. For each algorithm, the following results are reported: number of iterations (for FISTA and nmAPG, we include the iteration number that attains the lowest objective), the time the algorithm takes to run, the estimated values of \(\beta\) and \(\gamma\) for the lowest objective value rounded to five decimal places, the best attained objective, the value of the gradient at each direction, and the normalised \(\ell^2\) norm of the gradient \( \|\cdot\|_{2,r} \coloneqq \sV^{-\nicefrac 1 2} \|\cdot\|_2  \). The reported values for the objective and gradient are scaled down by \(\nicefrac{1}{\nn^2}\) which allows us to compute \eqref{eq:First-Objective} in terms of proportions of the total population.

\begin{table}[h]
\centering
\fontsize{9.5}{10.5}\selectfont
\setlength{\tabcolsep}{2.pt}
\def\arraystretch{1.4}
\begin{tabular}{ccc cc c ccc}
\toprule
\multirow{2}{*}{Algorithm} & {Iteration} & Time & \multicolumn{2}{c}{Best \(\alpha_{k^*}\)} & Objective & \multicolumn{3}{c}{Gradient} \\
\cline{4-5} \cline{7-9}
 & count  & (seconds)  &   \(\beta_{k^*}\) & \(\gamma_{k^*}\) & \(j(\alpha_{k^*})\)       &    \(\partial_\beta j(\alpha_{k^*})\) & \(\partial_\gamma j(\alpha_{k^*})\) & \(\| \nabla j(\alpha_{k^*}) \|_{2,r}\)       \\
\midrule
	PGD 	
	& 1697       & 27.37   & 0.03000 & 0.59998     	& \num{6.3 e-10} 	& \num{-3.7e-4} & \num{-5.1e-05}  & \num{2.7e-4}
	\\
         FISTA       
         & 363 (314)   & 2.66    & 0.02999 & 0.59999 & \num{1.9e-10} & \num{1.0e-5}       	& \num{-2.7e-5}   & \num{2.1e-5}       
    \\
        nmAPG
        & 294 (275)  & 4.69 & 0.03000 & 0.59998 & \num{1.4e-9} & \num{2.0e-3} & \num{-3.3e-5} & \num{1.4e-3}
         \\
         LM--BFGS 
			& 50  & 7.61   & 0.03000    & 0.59999  & \( \mathbf{ 1.0 \times 10^{-10}}\)   & \num{2.9e-5} & \num{-2.0e-5} & \num{2.5e-05}
	\\
\bottomrule
\end{tabular}
\caption{Optimisation results for the  optimal control problem \eqref{Optimal_Control_Problem} with cost functional \eqref{eq:First-Objective} applying PGD and \Cref{alg:FISTA-BT,alg:nmAPG-BT,alg:LMBFGS}. The index \(k^*\) is the one corresponding to the smallest objective value.
\review{The best objective value is highlighted in \textbf{bold}.}
}
\label{tab:ODE-Optimisation}
\end{table}

It is not surprising that PGD underperforms the other three algorithms in terms of number of iterations. FISTA follows with around one fourth of the iteration count, nmAPG yields a slight reduction, while LM--BFGS needs around one sixth of these. This can be understood by the linear convergence in terms of the iterates for PGD, FISTA, and nmAPG. We can observe the impact of the inner loops for nmAPG and LM--BFGS in terms of time, as the additional state systems take a toll on the performance of the algorithm.
Notwithstanding, the attained objectives of three out of four algorithms are of the same order of magnitude as well as their scaled gradients. This is due to the fact that the three methods are able to approximate \(\alpha^*\) up to five significant figures. By contrast, nmAPG still yields a close result but is not able to beat FISTA.
We can see this performance in \Cref{fig:Exact_Contours}. There we can observe that both FISTA and PGD have a good start, but as they approach the global minimum, their performance is damped by the linear convergence. For both plots, the reduced objective is computed around \num{4.5e+4} times, which is more than 25 times the number of iterates that PGD required to obtain a good approximation.

\begin{figure}[ht]
\centering
	\includegraphics[scale=0.37]{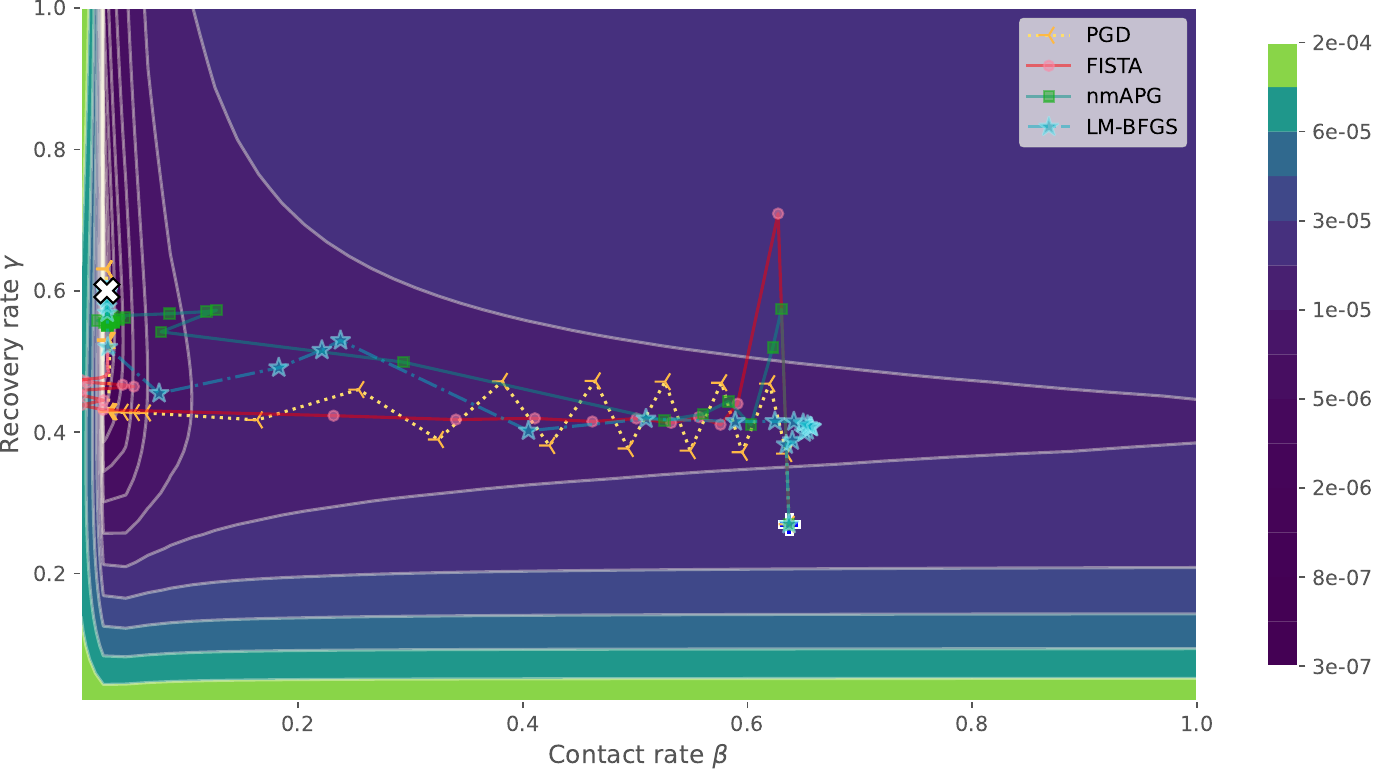}
	\includegraphics[scale=0.37]{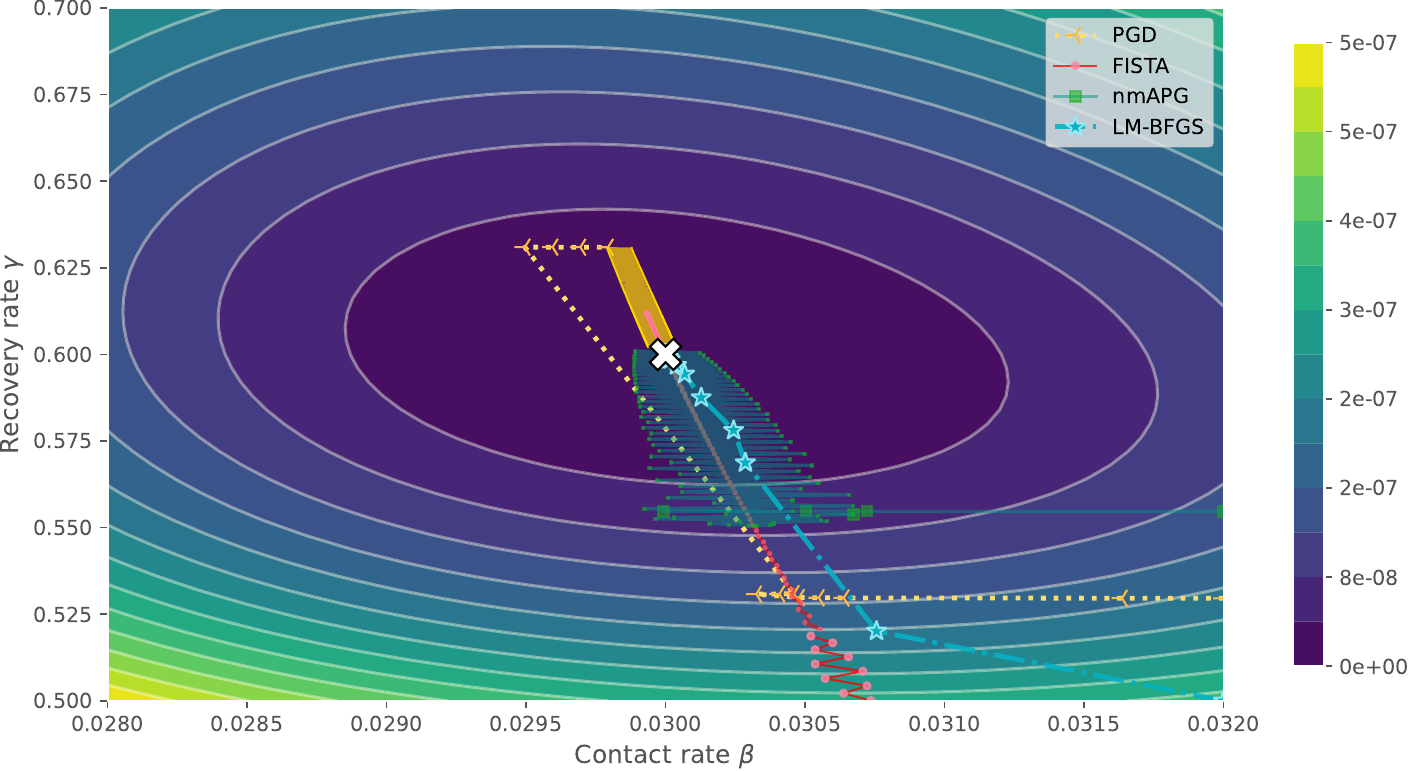}
\caption{Contour plots for the reduced objective \(j\) with known target \(\widehat{\rho}_a\). The panel on the left is the result of evaluating the scaled reduced objective function \(( \nicefrac{j(\alphaV)}{\nn^2} )\) over the set of feasible controls \(\AcalV = [0,1]^2\). 
The panel on the right zooms into the subregion \([ 0.028,0.032] \times [0.5, 0.7]\). 
In each panel, the optimisation path given by the sequence of iterates \( \{\alpha_{\mathrm{v},k}\}\) generated by each algorithm is illustrated. The \texttt{+} marker at the left panel is the starting point, while the \texttt{x} marker in both plots points at the optimal solution. }
\label{fig:Exact_Contours}
\vspace{-0.5\baselineskip}
\end{figure}

The evolution of the objective values at each iteration of the four algorithms is displayed in \Cref{fig:Exact_Costs}. We can observe how the quasi--Newton algorithm outperforms the first--order methods, attaining the smallest objective value found. All three first--order methods obtain a value of \(j\) below \( 10^{-2}\) quite quickly, yet the three algorithms stall after that and would take several hundreds of iterations to achieve a significant decrease in the objective function. 

\begin{figure}[htbp]
\centering
	\includegraphics[scale=0.55]{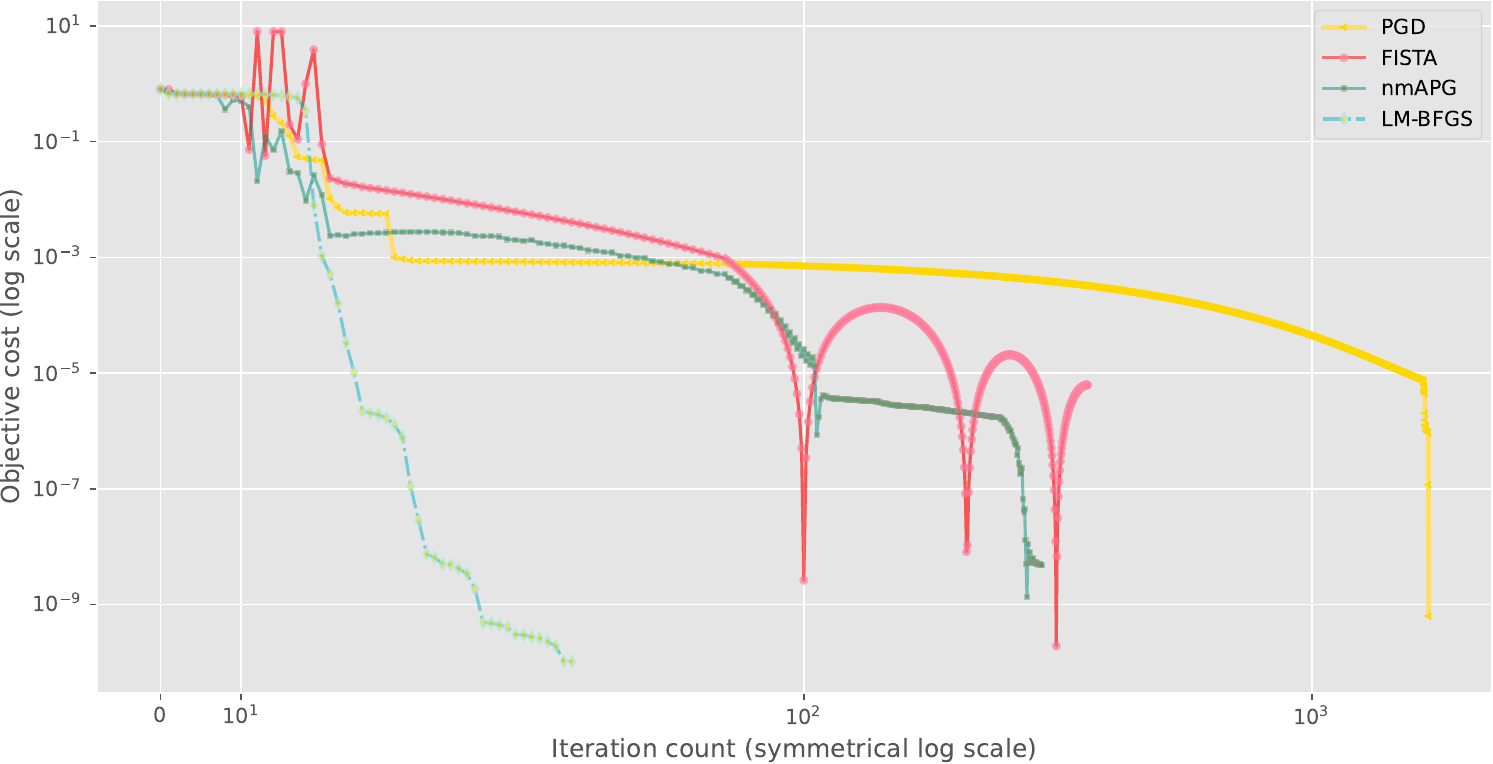}
\caption{Evolution of the reduced objective against number of iterations for each algorithm.}
\label{fig:Exact_Costs}
\vspace{-0.5\baselineskip}
\end{figure}

With this example, we have shown that the four algorithms solve the optimal control problem \eqref{Optimal_Control_Problem} effectively. They allow us to find good approximations of the optimal solution up to five decimal places (in most cases), and the number of required iterations is significantly smaller than that required with a grid search.

\subsection{Regularisation}

In this set of tests, we explore the rôle of the regularisation parameter when the existence of a unique minimiser is not apparent. Tikhonov regularisation consists of adding a penalty term to the running cost to penalise large values of \(\alpha\). It is often used in the context of control--constrained optimal control as it not only helps to prove the existence of a unique solution to problems like \eqref{Optimal_Control_Problem}, but also helps in the algorithmic design of numerical solution methods based on the gradient system \cite{Johansen1997,Daniels2017}.

For the base target, we obtain a solution of the state system \(\widetilde\rho_1\) associated to the known parameter  \( \alpha^* = \begin{psmallmatrix} \nicefrac{7}{100} & \nicefrac{1}{10} & \nicefrac{1}{20} \end{psmallmatrix}^\top \hspace{-0.3em}\) with initial conditions \(\rho_0 = \begin{psmallmatrix} 380 & 20 & 0 \end{psmallmatrix}^\top\), initial population \(\nn = 400\), and final time \(T = 3\). The choice of \(\alpha^*\) and \(\rho_0\) result in an fast epidemiological outbreak of the disease. As a disease of concern, initial estimates of the population in each compartment might be prone to error and re--collection of data is not instantaneous. We emulate this measuring issue by applying the smooth transform
\(
    \widetilde\rho_2 (t) \coloneqq  \widetilde\rho_1(t) + 4( \sin \widetilde\rho_1(t) - \sin \rho_0 ).
\)
Then we compute the rolling average of \( \widetilde\rho_2\)
on $k = 50$ uniform subintervals of $[0,T]$, yielding a vector \(p_2\). The resulting target is the semicontinuous piecewise constant function given by the entries of \(p_2\); i.e., 
\(
    \widehat{\rho}_{\text b} = p_3(t) = \sum\limits_{i=0}^{k-1} (p_2)_i \, \chi_{ \, t_{i} \leq t < t_{i+1}  } + (p_2)_{k-1} \, \chi_{t=T}.
\)

The sensitivity \eqref{eq:SIRD_Sensitivity} associated with the baseline parameter \(\alpha^*\) is given by the vector \( \Phi(\mathcal{R}_0 | \alpha^*) = \begin{psmallmatrix} 1 & -\nicefrac{2}{3} & -\nicefrac{1}{3}\end{psmallmatrix}^\top\). Thus, in this case the mortality rate has a significant impact on the evolution of the disease and so we select \( \AcalV = \Acal\) and \( \AcalF = \varnothing\).

In all tests, \(200\) Chebyshev points are used inside the interval \( [0,T]\), plus the two endpoints. 
The selected target suggests that an optimal fit must present an outbreak of the disease in finite time.
The initial guess \(\alpha_0\) is always set to \( \begin{psmallmatrix} \mathtt{0.07364913} & \mathtt{0.0184188} & \mathtt{0.03663371} \end{psmallmatrix}^\top \hspace{-0.3em}\). The selected initial approximation represents a challenge for the optimisation process as this point lies on a flat region of the objective. In contrast to the previous tests and the example in \Cref{fig:Convex_Contours}, the initial point lies in a flat valley away from a sharp minimum.
We set the tolerances for all methods to \( 10^{-13}\). The three first--order methods run for at most  \( \mathrm{it_{\max}} = 10\,000\) iterations, and LM--BFGS is allowed to restart every \(\mathrm{it_{\max}} = 100\) iterations.

Here we focus on the effect of the regularisation parameter \(\theta\) on the choice of the fit \(\rho\), which is controlled by the following objective:
\begin{equation}
\label{eq:Second-Objective}
	J(\rho,\alpha) \coloneqq \frac{1}{2\nn^2} \int\limits_0^T | \rS - \hrS |^2 + |\rI - \hrI|^2 + |\rR - \hrR|^2 \dif t + \frac{\theta}{2} \big[ \beta^2 + \gamma^2 + m^2 \big].
\end{equation}
This choice of payoff is a scaling of \eqref{eq:Second-Objective-r} in the population compartments, which effectively normalises the integral terms.
The associated gradient system is given by:
\begin{equation*}
	\alpha = \proj_{\Acal} \bigg\{ \alpha - \frac{\kappa}{\nn^2}
	\int\limits_0^T
	\begin{psmallmatrix}	
				\rS \rI (\qI - \qS ) 	\\
		\phantom{\rS} \rI (\qR - \qI)     \\
                -\rI \qI
	\end{psmallmatrix} 
	\dif t
 - \kappa \theta \alpha
	\bigg\}
    .
\end{equation*}
We would expect that whenever \(\theta\) is greater than 1, \(\alpha\) would be small and the state curves \(\rho\) would stop fitting the target \(\widehat{\rho}_{\text{b}}\). This is showcased in \Cref{fig:States_Regularisation}, where we include the state curves against each compartment against the target \(\widehat{\rho}_b\) for the best fit, which is obtained using nmAPG. 

\begin{figure}[htbp]
\centering
	\includegraphics[scale=0.4]{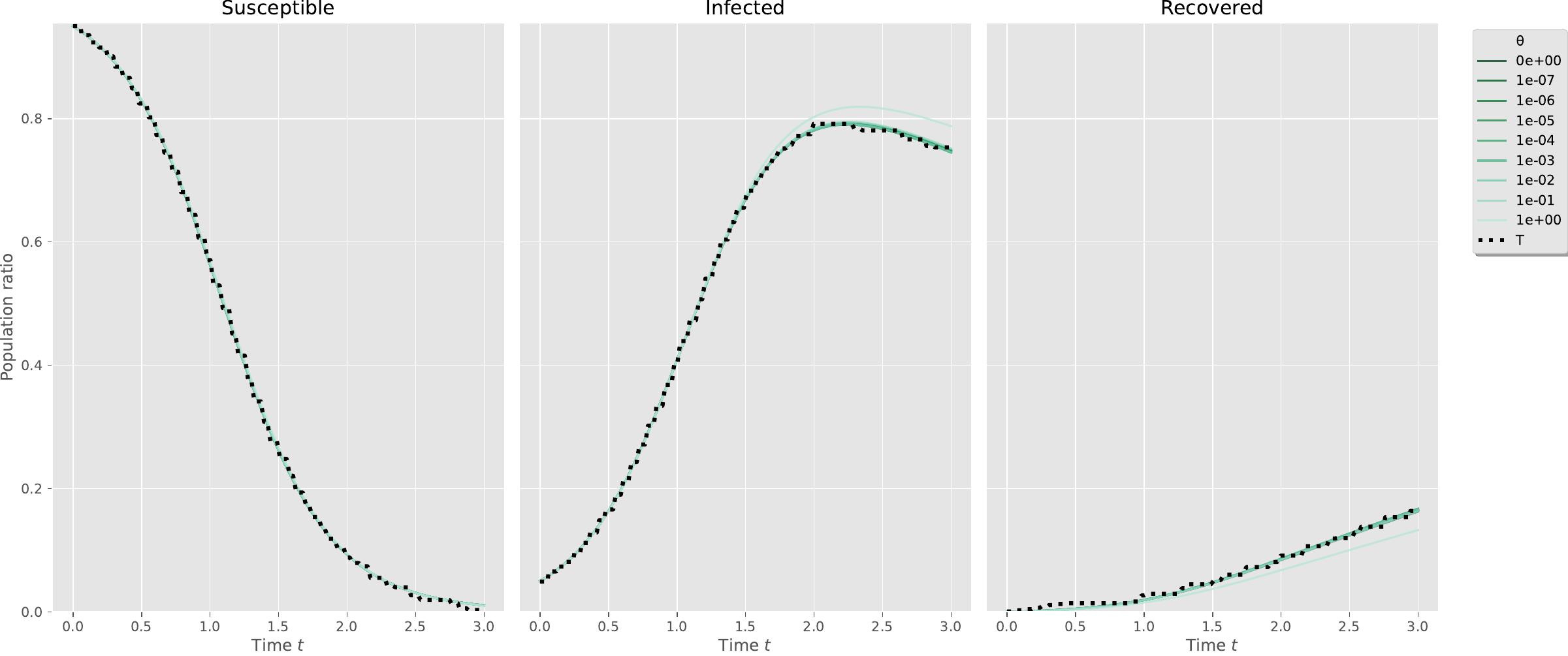}
\caption{State curves for different values of the regularisation parameter \(\theta\). The intensity of the colour of each curve is associated with the value of \(\theta\). Visually, most of the curves overlap, meaning an effective fitting of the data has been achieved.
The \review{dotted} curves in the three panels display the target \(\widehat{\rho}_b\). 
}
\label{fig:States_Regularisation}
\vspace{-0.5\baselineskip}
\end{figure}

In \Cref{tab:Regularisation-Results}, a summary of the results for PGD and \Cref{alg:FISTA-BT,alg:nmAPG-BT,alg:LMBFGS} is presented. The reported values show that FISTA and nmAPG attain similar controls and objective values. Note that we also test the value \(\theta = 0\). Overall, we observe that the best attained objective for this numerical experiment is very similar to the best objective attained for \(\theta \leq 10^{-3}\). As a result, we can conclude that we have reasonable flexibility in the selection of \(\theta\) when approximating the unregularised problem. Regarding PGD, we observe that the algorithm stagnates very early, hence reporting a small number of iterations but high objectives and controls that differ from the approximations obtained with the other two first--order algorithms. 
We observe that LM--BFGS displays a similar behaviour as most of the descent directions are gradient steps and the curvature information is relatively poor. \review{By considering alternative quasi--Newton approaches (e.g., \cite{Kanzow_24}), it is possible to achieve superior error metrics than nmAPG with a second--order method for this example \cite[\S 8]{AMT_Th_2024}.
Finally, we note that the values of the relative gradient norm are larger than for the previous experiment. During the optimisation process, we observe that \( \partial_\beta j (\alpha)\) produces the largest contribution to the gradient norm. This is understandable, as \(\beta\) has the largest sensitivity, and since its values are small, we can expect large gradient values in this direction, corresponding to the fact that the minimiser of \(j\) is sharp.}
%
%
%

\begin{table}[ht]
\centering
\fontsize{9.5}{10.5}\selectfont
\setlength{\tabcolsep}{2.pt}
\def\arraystretch{1.4}
\begin{tabular}{c ccc ccc c ccc}
\toprule
Regularisation & {Iterations} & Time & \multicolumn{3}{c}{Best \(\alpha_{k^*}\)} & Objective & Gradient \\
\cline{4-6} 
\(\theta\) & \(k^*\) & (seconds) &   \(\beta_{k^*}\) & \(\gamma_{k^*}\) & \(m_{k^*}\) & \(j(\alpha_{k^*})\)       &    \(\| \nabla j(\alpha_{k^*}) \|_{2,r}\)       \\
\midrule
	\multirow{4}{*}{0} & 
	11 & 0.24 & 0.006795 & 0.045854 & 0.058053 & 0.004794 & 0.286195
	\\
         & 
         683 & 12.98 & 0.007011 & 0.104485 & 0.050259 & 0.000182 & 0.007826
    \\
        & 
	412 & 13.83 & 0.007000 & 0.101248 & 0.049051 & \bf 0.000181 & 0.040392
         \\
         & 
	8 & 1.20 & 0.007048 & 0.092515 & 0.067470 & 0.000332 & 0.073289
	\\
	\hline
	\multirow{4}{*}{\(10^{-7}\)} & 
	11 & 0.25 & 0.006796 & 0.045860 & 0.058058 & 0.004793 & 0.343099
	\\
         & 
         531 & 10.62 & 0.007015 & 0.101494 & 0.048307 & 0.000186 & 1.017082
    \\
        & 
	488 & 18.74 & 0.007004 & 0.103370 & 0.048071 & \bf 0.000176 & 0.042219
         \\
         & 
	5 & 2.26 & 0.006997 & 0.084613 & 0.063076 & 0.000496 & 0.141659
	\\
	\hline
	\multirow{4}{*}{\(10^{-5}\)} & 
	11 & 0.25 & 0.006796 & 0.045861 & 0.058060 & 0.004792 & 0.358831
	\\
         & 
         972 & 17.37 & 0.007043 & 0.101730 & 0.053628 & 0.000191 & 0.290449
    \\
        & 
	464 & 17.62 & 0.006995 & 0.102149 & 0.047266 & \bf 0.000182 & 0.058861
         \\
         & 
	5 & 1.84 & 0.007050 & 0.090543 & 0.069421 & 0.000374 & 0.145930
	\\
	\hline
	\multirow{4}{*}{\(10^{-3}\)} & 
	11 & 0.25 & 0.006795 & 0.045862 & 0.058060 & 0.004795 & 0.318434
	\\
         & 
         520 & 10.15 & 0.006988 & 0.100474 & 0.047756 & 0.000199 & 0.903052
	\\ & 
    842 & 50.61 & 0.007007 & 0.103080 & 0.049019 & \bf 0.000182 & 0.029610
         \\
         & 
	58 & 3.68 & 0.007067 & 0.093155 & 0.071243 & 0.000393 & 0.092238
	\\
	\hline
	\multirow{4}{*}{\(10^{-1}\)} & 
	27 & 0.36 & 0.006802 & 0.045789 & 0.057926 & 0.005089 & 1.471839
	\\
         & 
         601 & 9.21 & 0.007036 & 0.094526 & 0.055433 & 0.000862 & 0.916192
    \\
        & 
	481 & 17.65 & 0.006990 & 0.101156 & 0.046518 & \bf 0.000815 & 0.060476
         \\
         & 
	9 & 4.34 & 0.007021 & 0.075050 & 0.075944 & 0.001464 & 0.075740
	\\
	\hline
	\multirow{4}{*}{\(1\)} & 
	12 & 0.21 & 0.006767 & 0.044791 & 0.056617 & 0.007728 & 0.221842
	\\
         & 
         594 & 8.22 & 0.006867 & 0.073698 & 0.043428 & 0.005488 & 1.085772
    \\
        & 
	3196 & 82.32 & 0.006889 & 0.080538 & 0.042268 & \bf 0.005388 & 0.641830
         \\
         & 
	71 & 2.34 & 0.006714 & 0.042604 & 0.046995 & 0.008392 & 0.116225
	\\
\bottomrule
\end{tabular}
\caption{Optimisation results for the  optimal control problem \eqref{Optimal_Control_Problem} with cost functional \eqref{eq:Second-Objective} applying PGD and \Cref{alg:FISTA-BT,alg:nmAPG-BT,alg:LMBFGS}. The index \(k^*\) is the one corresponding to the smallest objective value.
\review{The best objective values are highlighted in \textbf{bold}.}
}
\label{tab:Regularisation-Results}
\end{table}

With this example, we have showcased the effects of the regularisation parameter \(\theta\) when solving the optimal control problem \eqref{Optimal_Control_Problem} with the objective function \eqref{eq:Second-Objective}. We have observed a robust behaviour in terms of the controls and objective values such that the attained minimisers approximate the minimisers of the unregularised problem. Moreover, we observe that as \(\theta\) grows, the controls become smaller at the expense of the fit. As a final note, we mention that we could select a different cost or penalty term in the objective function for each coordinate of the control. However, this change would not drastically change our results as these weights should be small to allow good fitting of the state variables.

\subsection{Data--driven parameter identification}\label{sec:DD_PI}

For our last set of experiments, we study the fitting of the SIRD model against COVID--19 case data from Singapore. 
In particular, we use the data from \cite{Mu2023} that fits a modified SEIR model which includes quarantining, provision of vaccines, and booster shots. 
We focus on the time period March 1 to April 30, 2020, which was before the widespread setup of vaccination campaigns. 
For this, we merge the Infected compartment subcategories (asymptomatic, mild, and severe). We caution that this choice, as well as the limitations of the data available at the start of the pandemic, results in a less general model; this is designed solely as a proof--of--concept for applying our methodology to real--world data.

As time evolved, better measurements were available, hence it makes sense to also emphasise the fitting at the final time. As a result, we can consider the following objective with a terminal cost:
\begin{equation}
\label{eq:Third-Objective-a}
	\widehat{J}(\rho,\alphaV) \coloneqq \frac{1}{2} \int\limits_0^T \| \rho(t) -  \widehat{\rho}_{\text c} (t) \|_{\R^3}^2 \dif t + \frac{1}{2} \| \theta \review{\,\odot\,} \alphaV \|^2_{\R^{s_{\textrm{v}}} } + \frac{1}{2} \| \vartheta \review{\,\odot\,} \big(\rho(T) -  \widehat{\rho}_{\text c} (T) \big) \|^2_{\R^3},
\end{equation}
where \( \theta, \vartheta \in \R^{s_\mathsf{v}}_{\geq 0} \) are regularisation vectors\review{, and \(\odot\) again denotes the Hadamard product}.
In this case, we assume that no parameter information is known besides the following: (a) social contact was low due to distancing advice, (b) there was possible time dependence in the parameters, (c) there were no significant population changes due to additional factors (e.g., additional births, fatalities, or migration).
The baseline assumption leads us to consider the choice \( \AcalV = \Acal\) and \( \AcalF = \varnothing\).
Assumption (a) allows us to limit the search space of \(\beta\) to the smaller set \( [0,10^{-2}]\). Additionally, assumption (b) suggests we extend our parameter space to measurable functions in \( L^\infty(0,T) \). Hence we consider 
\(
    \Acal = \{ \beta \in L^2(0,T): 0 \leq \beta \leq 10^{-2} \} \times \{ (\gamma,m) \in [L^2(0,T)]^2: 0 \leq \gamma,m \leq 1 \}.
\)
In such case, most of the mathematical analysis that we presented in the previous sections follows by applying Carathéodory theory for ODEs, although solutions can certainly display less regularity. Even though the adjoint system remains the same, the gradient system now displays pointwise information to be understood in a positive--measure (or \emph{almost everywhere}) sense:
\begin{equation*}
	\alpha(t) = \proj_{\Acal} \bigg\{ \alpha(t) - \kappa
	\begin{psmallmatrix}	
				\rS \rI (\qI - \qS ) 	\\
		\phantom{\rS} \rI (\qR - \qI)     \\
                -\rI \qI
	\end{psmallmatrix} 
 - \kappa \pd{r}{\alpha}(t)
	\bigg\}
    ,
\end{equation*}
where \(r\) is the runoff cost associated with \eqref{eq:Third-Objective-a}; see \eqref{eq:Third-Objective-r}. 
Although the use of time--dependent controls can provide a more complete understanding of the underlying mechanisms of contact and changes in recovery and mortality, possible bang--bang solutions can result in parameters that lack epidemiological significance. In particular, we should prevent controls for which $\gamma+m > 1$. Within our framework, we can do so by using an additional penalty function $H$ that continuously penalises pairs of controls as follows: first setting $H_p(s) = \max\{0, \operatorname{sign}(s) s^2\}$, which has a continuous derivative $H_p'(s) = \max\{0, 2s\} $, we define $H(\alpha) = \upsilon H_p( \gamma + m - 1 )$ for a penalty parameter $\upsilon > 0$. Then, we consider the following tracking cost:
\begin{equation}
\label{eq:Third-Objective-b}
	J(\rho,\alpha) \coloneqq  \int\limits_0^T \frac{1}{2} \| \rho(t) -  \widehat{\rho}_{\text c} (t) \|_{\R^3}^2 
    + \frac{1}{2} \| \theta \review{\,\odot\,} \alpha (t) \|^2_{\R^3} 
    + H\big(\alpha(t)\big) \dif t + \frac{1}{2} \| \vartheta \review{\,\odot\,} \big(\rho(T) -  \widehat{\rho}_{\text c} (T) \big) \|^2_{\R^3},
\end{equation}
which also yields
\(
    \pd{r}{\alpha}(t) = \theta^2 \review{\,\odot\,} \alpha(t) + 2\upsilon \max\{0,\gamma(t) + m(t) - 1\} 
    \begin{psmallmatrix}
        0 & 1 & 1
    \end{psmallmatrix}^\top \hspace{-0.3em}.
\)
Notice that the boundedness of \(\alpha\) also implies that \(H\) and its derivative are bounded, hence the existence theory remains unchanged.
Moreover, the nonzero terminal condition of the adjoint is given by \( q(T) = \vartheta^2 \review{\,\odot\,} \big[ \rho(T) - \widehat{\rho} (T) \big] \).

In all tests, for illustrative purposes we scale time by the number of weeks associated with the data, i.e., \(T \approx 8.57\), and scale the total population by tens of thousands. The selected starting date yields the initial (scaled) population \( \nn \approx 585\). 
Once more, \(200\) Chebyshev points were used inside the interval \( [0,T]\), plus the two endpoints. 
We set the tolerances for all methods to \( 10^{-13}\). 
The three first--order methods run for at most  \( \mathrm{it_{\max}} = 5\,000\) iterations, and LM--BFGS is allowed to restart every \(\mathrm{it_{\max}} = 100\) iterations.
For the regularisation weights, we considered the vectors \( \theta^2 = \begin{psmallmatrix}  \num{e-6} & \num{e-8} & \num{e-9}  \end{psmallmatrix}^\top \hspace{-0.3em} \)  and \( \vartheta^2 = \begin{psmallmatrix}  \num{e-4} & \num{e-4} & \num{5e+2}  \end{psmallmatrix}^\top  \hspace{-0.3em} \),
proportionally scaled with respect to \(\nn\). The former choice reflects on the expected order of magnitude of each parameter, while the latter puts emphasis on penalising deviations from the recovered compartment.

\begin{table}[h]
\centering
\fontsize{9.5}{10.5}\selectfont
\setlength{\tabcolsep}{2.pt}
\def\arraystretch{1.4}
\begin{tabular}{ccc cc c ccc}
\toprule
\multirow{2}{*}{Algorithm} & {Iteration} & Time & Objective & Gradient \\
 & count  & (seconds)  &   \(j(\alpha_{k^*})\)       &    \(\| \nabla j(\alpha_{k^*}) \|_{2,r}\)       \\
\midrule
	PGD 	
	& 89       & 1.02   & \num{3.3 e-6} 	& \num{4.5e-4} 
	\\
    FISTA       
    & 772 (625)   & 11.69    & \num{9.8e-07} & \num{3.9e-4}
    \\
    nmAPG
    & 1020 (648)  & 22.28 & \(\mathbf{5.5 \times 10^{-8}}\) & \num{2.9e-4}
    \\
    LM--BFGS 
    & 55  & 15.51   & \num{2.0e-07}   & \num{3.8e-4}
	\\
\bottomrule
\end{tabular}
\caption{Optimisation results for the  optimal control problem \eqref{Optimal_Control_Problem} with cost functional \eqref{eq:Third-Objective-b} applying PGD and \Cref{alg:FISTA-BT,alg:nmAPG-BT,alg:LMBFGS}. The index \(k^*\) is the one corresponding to the smallest objective value.
\review{The best objective value is highlighted in \textbf{bold}.}
}
\label{tab:ODE-Optimisation-time-control}
\end{table}


The results of the fitting are presented in \cref{tab:ODE-Optimisation-time-control}. We see that nmAPG is able to obtain the best objective value. Although the resulting iteration count for nmAPG is the largest, the computational time is still less than a minute. 
LM--BFGS stalls due to its globalisation strategy (the gradient step), but obtains a reasonably good solution and the second best objective value. 
Moreover, the objective values do not present any active penalisation terms related to \(H\), and the values of the controls for all algorithms were of similar orders of magnitude.

The best fitting, with respect to the objective cost and in the original units of data recollection, is presented in  \cref{fig:Time_Dependent}.
Here, we observe that the fitted curves follow a similar trajectory as the data. 
Along these lines, we observe that the controls follow suit: as more fatalities are registered, the mortality rate increases, while the recovery rate decreases. 
The contact rate remains almost constant until near the final time when we observe a significant decrease. A possible interpretation is that other non--pharmaceutical measures (e.g., mask usage and social distancing) were imposed during the analysed timeframe. 
Additionally, observe that the recovery rate implied by the data is substantially lower than the likely \emph{true rate} as the reported sizes of the Infected and Recovered compartments are much lower than the ``\emph{true}’’ underlying values, hence the fitting suggests higher mortality (also lower transmission rates) than the actual values associated with the disease.
We also observe that the recovered curve is the furthest from the shape of the registered data. This can be understood from the following observations: (a) real epidemics are expected to follow SIR--based trajectories only locally, (b) several subcompartments of clinical data (that evolve at different time scales) have been grouped into the infected compartment, (c) the data is likely to be statistically incomplete in the sense that information related to the infected and recovered compartments was inaccurate at the start of the pandemic. 
As a result, considering that a perfect fit implicitly assumes accuracy in the data, as well as assuming a perfect model, we cannot expect an exact re--creation. 
On the other hand, the results of the fitting are able to explain the quantitative behaviour of the population model, and the approach can certainly be used to fit other more complete compartmental models. Investigating how to adjust the model to incomplete and/or unreliable data would be a valuable avenue of future work.

\begin{figure}[htbp]
\centering
	\includegraphics[scale=0.4]{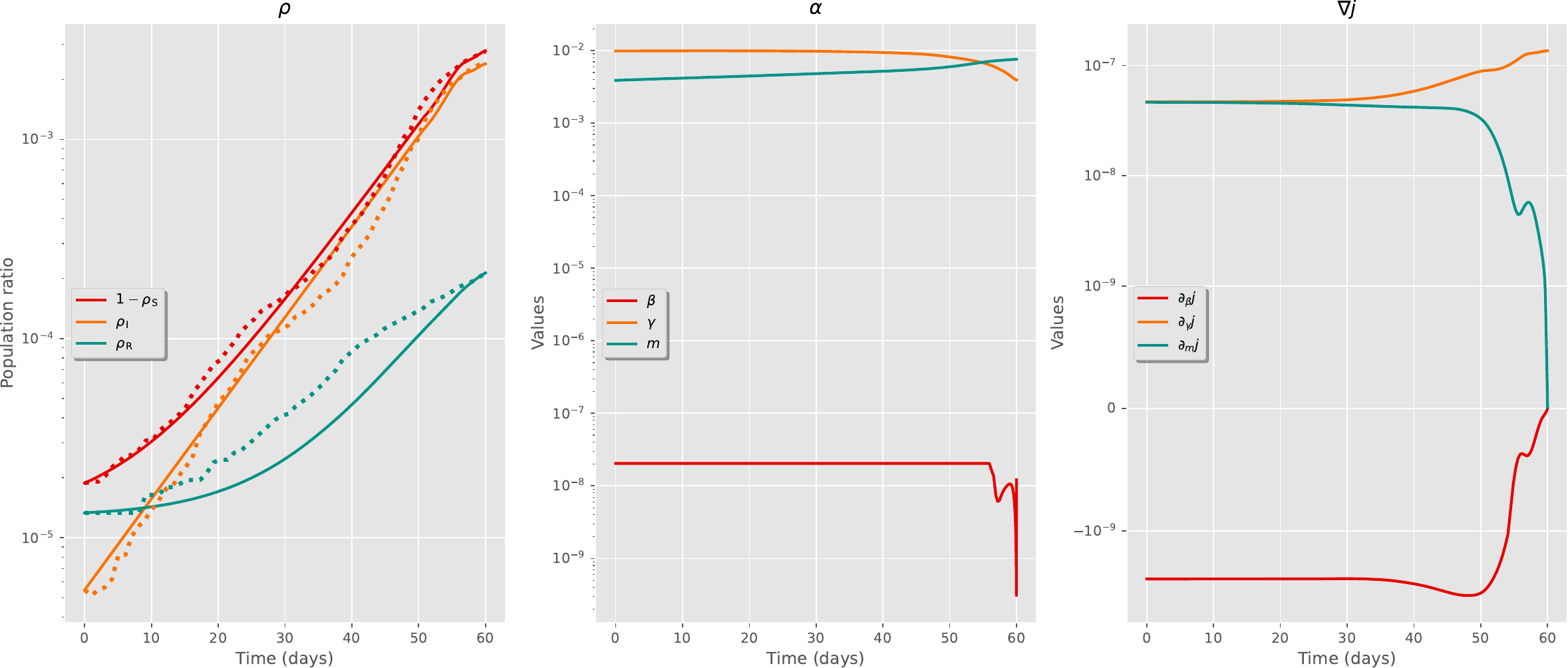}
\caption{Results of the fitting. In all plots, time has been scaled to daily information and \(\alpha\) has been scaled to real population units. \textbf{First panel:} State curves (continuous lines) of best fit against data (dotted line). We report \(1-\rS\) for visualisation purposes. All curves were normalised by \(\nn\) and are displayed in logarithmic scale. \textbf{Second panel:} Obtained controls (in logarithmic scale).  \textbf{Third panel:} Gradient curves associated with each time--dependent parameter.
}
\label{fig:Time_Dependent}
\end{figure}
%


\section{Conclusions}
\label{sec:conclusions}

In this work we have developed a systematic treatment of a parameter identification problem for the SIRD model \eqref{Optimal_Control_Problem}. We began studying the well--posedness of the forward problem \eqref{eq:ode-sys}, proving existence of solutions for fixed parameters in \Cref{th:state_existence} and qualitative properties such as non--negativity and boundedness in \Cref{Th:ODE_Qualitative}. We then demonstrated the existence of optimal calibration parameters in \Cref{Th:Admissibility}, which allow SIRD--based models of the form \eqref{sys:SIRD} to be parameterised based on target, or experimental, data. The next key development was a practical methodology to determine these optimal parameters. To provide this, we took a general approach in deriving first--order optimality conditions for problem \eqref{Optimal_Control_Problem}, described by the state, adjoint, and gradient systems in \Cref{th:Optimality-System}.  In principle, this route can also be taken with more complex compartmental systems that are of a similar form to \eqref{eq:ode-sys}, such as those in \cite{Cheng2023,Elzinga2023,Moffett2023,Penn2023,Viscardi2023,CuevasMaraver2024,Barua2023}.  This would be an interesting topic of future work.

In order to address such problems in practice, we required an accurate, robust, and efficient numerical scheme to solve the problem \eqref{Optimal_Control_Problem}. In \Cref{Sec:Algorithms}, we described four such optimisation algorithms: PGD, FISTA, nmAPG, and LM--BFGS. PGD and FISTA are easy--to--implement methods that can be used to obtain fast approximations, but the lack of convexity in the reduced cost functional limited the performance of the algorithms for the problems studied here. By contrast, nmAPG and LM--BFGS are more specialised optimisation methods, which require more effort to implement.  In the examples studied here we found that LM--BFGS generally outperformed both PGD and FISTA in terms of minimising the objective function and the gradient. The nmAPG approach was found to be competitive with LM--BFGS, although it required many more iterations, so it may not be an appropriate choice when the cost of each iteration is larger than for problems studied here. Overall, the numerical experiments in \Cref{sec:numerics} encompassed a range of examples that displayed the robustness of our method for parameter identification, and described effective optimisation methods for resolving such problems numerically (which can be tailored to individual problems, including different compartmental systems). Thus, we have provided a pipeline that enables the parametrisation of SIR--type models from data, significantly increasing their potential impact and reducing the need for \emph{ad hoc} modelling choices. 

\review{Although most of the theoretical framework developed in this paper focuses on the study of time--independent parameters, many of our results can be extended for the time--dependent case by means of Carathéodory theory for ODEs and optimisation theory in Banach spaces. In particular, the state and adjoint systems remain unchanged (understanding the involved parameters as time--dependent quantities); however, the gradient system can be described pointwise (almost everywhere) by a \emph{local} operator. The use of local information can ease the evaluation of numerical algorithms (as the absence of nonlocalities allows for alternative numerical approaches). Notwithstanding, such extensions require an additional understanding of the possible variability of the parameters, as the emergence of bang--bang controls can limit the biological significance of the optimised controls. This could be an exciting avenue for future work.}

\review{Further} extensions of this work, beyond increasing the complexity of the system of ODEs, could involve introducing spatial information, such as diffusion or attractive and repelling forces between agents \cite{Vrugt2021,Vrugt2020}. We believe the optimisation framework devised in this work can provide a guide to the efficient numerical simulation of more sophisticated applications such as these.

\section*{Data, code and materials}
Code for the experiments and implemented models is available at
\begin{center}
    \noindent \href{https://github.com/andresrmt/SIRD_Control}{\texttt{https://github.com/andresrmt/SIRD\_Control}}
\end{center}





\begin{small}
\bibliographystyle{elsarticle-num-names}
\bibliography{Main.bbl}
\end{small}

\end{document}